\documentclass[11pt, reqno]{amsart}
\usepackage{hyperref}
\usepackage{cmap}                        
\usepackage[utf8]{inputenc}            
\usepackage[russian,english]{babel}
\usepackage[left=1in,right=1in,top=1in,bottom=1.5in]{geometry} 
\usepackage{amssymb,amsmath, amsthm, amscd,ifthen}
\usepackage{graphicx}
\usepackage[svgnames]{xcolor}
\usepackage{tikz}
\usetikzlibrary{patterns,decorations.text,positioning,arrows,shapes,decorations.markings,shapes.geometric}
\tikzstyle{vertex}=[circle,draw=black,fill=black,inner sep=0,minimum size=0.2cm,text=white,font=\footnotesize]

\newtheorem*{thm*}{Theorem}
\newcommand{\ff}{{\mathcal F}}
\newcommand{\aaa}{{\mathcal A}}

\newcommand{\G}{{\mathcal G}}
\newcommand{\h}{{\mathcal H}}

\newtheorem*{gypoerd}{Erd\H os Matching Conjecture}

\newtheorem*{cla*}{Claim}
\newcommand{\bb}{{\mathcal B}}

\newtheorem{thm}{Theorem}

\newtheorem{lem}[thm]{Lemma}
\newtheorem{cla}[thm]{Claim}
\newcommand{\eps}{{\varepsilon}}

\newtheorem{cor}[thm]{Corollary}
\date{}

\newtheorem{prb}{Problem}
\newtheorem{prop}[thm]{Proposition}

\DeclareMathOperator{\E}{\mathrm E}

\newcommand{\Var}{\mathrm{Var}}
\newcommand{\Cov}{\mathrm{Cov}}

\date{}
\title{The Erd\H os Matching Conjecture and Concentration Inequalities}
\author{Peter Frankl}\address{R\'enyi Institute, Budapest, Hungary; Email: {\tt peter.frankl@gmail.com}}

\author{Andrey Kupavskii}
\address{G-SCOP, CNRS, Universit\'e Grenoble-Alpes, France and
Moscow Institute of Physics and Technology; Email: {\tt kupavskii@ya.ru}.} \thanks{The research of the second author was partially supported by the Advanced Postdoc.Mobility grant no. P300P2\_177839 of the Swiss National Science Foundation, EPSRC grant no. EP/N019504/1, by the grant 18-01-00355 of the Russian Foundation for Basic Research, and the grant of the president \foreignlanguage{russian}{НШ}-2540.2020.1}

\begin{document}
\maketitle
\begin{abstract}
  More than 50 years ago, Erd\H os asked the following question: what is the maximum size of family $\ff$ of $k$-element subsets of an $n$-element set if it has no $s+1$ pairwise disjoint sets? This question attracted a lot of attention recently, in particular due to its connection to various combinatorial, probabilistic and theoretical computer science problems. Improving the previous best bound due to the first author, we prove that $|\ff|\le {n\choose k}-{n-s\choose k}$, provided $n\ge \frac 53sk -\frac 23 s$ and $s$ is sufficiently large. The bound on $|\ff|$ is sharp since the family of all $k$-sets that intersect some fixed $s$-element set has such size and has no $s+1$ pairwise disjoint sets. We derive several corollaries concerning Dirac thresholds and deviations of sums of random variables. We also obtain several related results.
\end{abstract}

\section{Introduction}
We consider the following classical problem due to Erd\H os.  Suppose that positive integers $n,k,s$ satisfy $n\ge k(s+1)$. Let $\ff\subset{[n]\choose k}$ be a $k$-graph (a family of $k$-element subsets) on the vertex set $[n]:=\{1,\ldots,n\}$. A {\it matching} in $\ff$ is a collection of pairwise disjoint sets in $\ff$. We denote by $\nu(\ff)$ the {\it matching number} of $\ff$, that is, the maximum size of a matching in $\ff$. Then the problem is as follows: determine the maximum $m(n,k,s)$ of $|\ff|$ subject to the condition  $\nu(\ff)<s+1.$

Each of the following families has matching number $s$.
\begin{equation}\label{ainks}\mathcal A_i(n,k,s):=\Big\{A\in {[n]\choose k}: |A\cap [i(s+1)-1]|\ge i\Big\}.\end{equation}
Let us put $\mathcal A:=\mathcal A_1(n,k,s)$ and $\mathcal A_k:=\mathcal A_k(n,k,s)$ for shorthand. 
Note that
\begin{eqnarray}
  \label{size1} |\aaa| & =&{n\choose k}-{n-s\choose k}={n-1\choose k-1}+\ldots +{n-s\choose k-1} \ \ \ \text{and}\\
  \label{size2} |\aaa_k| & =&{k(s+1)-1\choose k}.
\end{eqnarray}
Note also that the right hand side of \eqref{size2} is independent of $n$.

Erd\H os conjectured that $m(n,k,t)$ always equals the right hand size of either \eqref{size1} or \eqref{size2}.
\begin{gypoerd}[Erd\H os, \cite{E}]
  We have \begin{equation}\label{eqemc0} m(n,k,s)=\max\Big\{{n\choose k}-{n-s\choose k}, {k(s+1)-1\choose k}\Big\}.
          \end{equation}
\end{gypoerd}
It was one of the favourite problems of Erd\H os and there was hardly a combinatorial lecture of him where he did not mention it.

The Erd\H os Matching Conjecture, or EMC for short, is trivial for $k=1$ and was proved by Erd\H os and Gallai \cite{EG} for $k=2$. It was settled in the case $k=3$ and $n\ge 4s$ in \cite{FRR}, for $k=3$, all $n$ and $s\ge s_0$ in \cite{LM}, and, finally, it was completely resolved for $k=3$ in \cite{F6}.

The case $s=1$ is the classical Erd\H os-Ko-Rado theorem \cite{EKR} which was the starting point of a large part of ongoing research in extremal set theory.

In his original paper, Erd\H os proved \eqref{eqemc0} for $n\ge n_0(k,s)$ for some $n_0(k,s)$. His result was sharpened by Bollob\'as, Daykin and Erd\H os \cite{BDE}, who established \eqref{eqemc0} for $n\ge 2k^3s$. Subsequently, Hao, Loh and Sudakov \cite{HLS} proved the EMC for $n\ge 3k^2s$. Their proof relies in part on the ``multipartite version'' of the following universal bound from \cite{Fra3}:
\begin{equation}\label{eqemcbad}
  m(n,k,s)\le s{n-1\choose k-1}.
\end{equation}
If $n= k(s+1)$ then the right hand side of \eqref{eqemcbad} is equal to $|\aaa_k|$. For this case, the EMC was implicitly proved by Kleitman \cite{Kl}. This was extended very recently by the first author \cite{F7}, who showed that $m(n,k,s)\le {k(s+1)-1\choose k}$ for all $n\le (s+1)(k+\eps)$, where $\eps$ depends on $k$. The first author \cite{F4} also proved \eqref{eqemc0} for $$n\ge (2s+1)k-s.$$ An easy computation shows that $|\aaa|>|\aaa_k|$ already for $n\ge (k+1)s$, that is, $m(n,k,s)={n\choose k}-{n-s\choose k}$ should hold also for $(k+1)s<n<(2s+1)k-s$. The aim of the present paper is to prove the following.
\begin{thm}\label{thm1} There exists an absolute constant $s_0$, such that
\begin{equation}\label{eqemc}m(n,k,s)= {n\choose k}-{n-s\choose k}\end{equation}
holds if $n\ge \frac 53 sk-\frac 23s$ and $s\ge s_0$.
\end{thm}

Roughly speaking, Theorem~\ref{thm1} settles the EMC for $1/3$ of the cases left over by \cite{F4}. We believe that the EMC is one the most important open problems in extremal set theory, playing a major role in several extremal problems in combinatorics. At the same time, its importance goes beyond combinatorics. As it was pointed out in \cite{aletal}, \cite{AHS}, it is deeply related to certain problems in probability theory  dealing with generalizations of Markov's inequality, as well as some computer science questions.

In the remaining part of this section we shall discuss universal bounds for the EMC improving \eqref{eqemcbad}. We discuss problems related to the EMC in Section~\ref{secappl}. In Section~\ref{prel} we give necessary preliminaries, mostly related to shifting and shadows.
In Section~\ref{secknez} we state and prove results on the concentration of intersections of families and random matchings. In Section~\ref{sec2} we give the proof of Theorem~\ref{thm1}. One of the Lemmas used in Section~\ref{sec2} has a technical proof, and it is deferred to the Appendix.

\subsection{Bounds for the Erd\H os Matching Conjecture}\label{sec15}
We have so far seen only one bound on $m(n,k,s)$: the bound \eqref{eqemcbad}. Another bound \eqref{eqemcgen2}, which works for large $s$ and fixed $k$, and which is good for matchings that are nearly perfect, is given in Section~\ref{secappl}. Exploiting the approach of Frankl \cite{F4}, Han \cite{Han16} proved the following global bound for the EMC, valid for $1<\gamma\le 2-1/k$.
\begin{equation}\label{eqemcgen}
m(n,k,s)\le {n\choose k}-{n-s\choose k}+\frac{(2-\gamma)k-1}{\gamma k-1}s{n-s-1\choose k-1} \ \ \ \ \text{for }\ \ n\ge \gamma k(s+1)+k-1.
\end{equation}
For $\gamma=2-1/k$, one recovers the original bound of Frankl \cite{F4}, while for $\gamma=1$ we get a trivial bound $m(n,k,s)\le {n\choose k}$.

In this paper, we prove the following universal bound for the EMC.

\begin{thm}\label{thmemcgen}
  Fix some $1<\gamma\le \frac 53$. Then there exists $s_0$, such that the following holds for any $s\ge s_0$. If $n\ge \gamma sk-(\gamma-1)s$ then \begin{equation}\label{eqemcgen3} m(n,k,s)\le {n\choose k}-\frac{\gamma-1}2\cdot\frac{5k-2}{\gamma k-(\gamma-1)}{n-s\choose k}.\end{equation}
\end{thm}
The proof uses Theorem~\ref{thm1} as a black box, and is given in Section~\ref{secgen}. The bound \eqref{eqemcgen3} is weaker than \eqref{eqemcgen} for roughly $\gamma \le 4/3$, and is stronger for $\gamma \ge 4/3$. 
The approach that gives Theorem~\ref{thmemcgen} is potentially very useful, as it relates the problem on the number of edges in Kneser graphs and the EMC (and may be combined with any bound for $m(n,k,s)$).  This theorem has applications to Dirac thresholds (see Proposition~\ref{prop33} in Section~\ref{secappl}).
\section{Preliminaries}\label{prel}

First let us give a simple proof of \eqref{eqemcbad} in case $n=qk$, where $q$ is an integer larger than $s$. Let $[n]=A_1\sqcup\ldots \sqcup A_q$ be any {\it full partition} of $[n]$ into $k$-sets. Let $\ff\subset {[n]\choose k}$ satisfy $\nu(\ff)\le s$. Then
\begin{equation}\label{eqsimple}|\ff\cap \{A_1,\ldots, A_q\}|\le s\end{equation} should be obvious.
By the Baranyai Theorem \cite{Baran}, for $n=qk$ one can partition ${[n]\choose k}$ into ${n-1\choose k-1}$ full $k$-set partitions. Then \eqref{eqsimple} implies \eqref{eqemcbad}.

As a matter of fact, using a bit of probability one can circumvent the use of Baranyai Theorem.
Namely, choose a full partition at random from the uniform distribution over all full partitions. Then for $i\in[q]$, we have $\Pr[A_i\in \ff]=|\ff|/{n\choose k}$.
By additivity of expectation, $\mathrm E[|\{A_1,\ldots, A_q\}\cap \ff|]=q|\ff|/{n\choose k}$.
By \eqref{eqsimple}, the left hand side is never more than $s$, thus we have
$$|\ff|\le {n\choose k}\frac sq = s{n-1\choose k-1} \ \ \ \ \ \text{for }n=qk.$$
The reason that we presented this simple argument is two-fold. Firstly, it is easy to understand. Secondly, investigating the size of the intersection of a fixed family $\ff\subset{[tl]\choose l}$ with randomly chosen full partitions $\{A_1,\ldots, A_t\}$ is the main new ingredient of our proof. We present two bounds (Lemma~\ref{lemcheb} and Theorem~\ref{thmazuma}) showing that the size of this intersection is concentrated around its mean. Let us mention that the proof is due to the second author.
Both bounds exploit the eigenvalue properties of Kneser graphs via a result of Alon and Chung \cite{AC}. The first concentration result uses Chebyshev's inequality, while the second is based on the Azuma--Hoeffding inequality (\cite{Azu67}) for martingales. Hopefully, both these bounds will prove useful in other situations as well.


The main combinatorial ingredients of the proof of the main theorem (Theorem~\ref{thm1}) are related to shifting, an operation invented by Erd\H os, Ko and Rado \cite{EKR}. It was first used in the context of the EMC in \cite{Fra3}. Let us state it in the form that we are going to use it.

Let $(a_1,\ldots, a_k)$ stand for a $k$-set with $a_1<\ldots<a_k$. The so-called {\it shifting partial order} is defined on $k$-sets, and we say that $(a_1,\ldots, a_k)$ precedes $(b_1,\ldots, b_k)$ if $a_i\le b_i$ for all $i\in[k]$ and the two $k$-sets are distinct.

One can define this for unordered sets $A,B$  by simply comparing their elements after ordering them increasingly. Let $A\prec B$ denote the fact that $A$ precedes $B$ in the shifting partial order.

A family $\ff\subset {[m]\choose k}$ is called {\it initial} ({\it shifted}) if $G\prec F\in \ff$ implies $G\in \ff$.

\begin{lem}[\cite{Fra3}]\label{lemfra}
  For every family $\tilde \ff\subset {[m]\choose k},$ there is an initial family $\ff\subset{[m]\choose k}$ with $|\ff|=|\tilde \ff|$ and $\nu(\ff)\le \nu(\tilde \ff)$.
\end{lem}
In view of this lemma, we can restrict our investigation to initial families when dealing with the EMC.

\begin{prop}\label{prop44} If $\ff\subset{[m]\choose k}$ is initial and $\nu(\ff)\le s$ then   $(s+1,2(s+1),\ldots, k(s+1))\notin\ff$. Consequently, for every $F\in \ff$, there exists some $i$, $1\le i\le k$, such that
  \begin{equation}\label{eqshift3}
    |F\cap [i(s+1)-1]|\ge i.
  \end{equation}
\end{prop}
\begin{proof}
  Define $F_i=(i,i+(s+1),\ldots, i+(k-1)(s+1))$. Should $F_{s+1}\in \ff$ hold, $F_i\prec F_{s+1}$ would imply $F_i\in \ff$ for  $1\le i\le s$ as well. However, $F_1,\ldots, F_{s+1}$ are pairwise disjoint, contradicting $\nu(\ff)\le s$. As for the ``consequently'' part, the opposite is equivalent to $(s+1,\ldots, k(s+1))\prec F.$
\end{proof}
For a family $\G\subset {[m]\choose k}$ let $\partial \G$ denote its {\it immediate shadow}:
\begin{align*}\partial \G:=\big\{H:|H|=k-1, \exists G\in \G \text{ s.t. } H\subset G\big\} \ \ \ \ \ \ \text{or, equivalently, }\ \ \ \ \ \
\partial \G:= \bigcup_{G\in \G} {G\choose k-1}.
\end{align*}


For any $S\subset[s+1]$, define the family $\ff(S)$ by
$$\ff(S):=\{F-S: F\in \ff,F\cap [s+1]=S\}.$$
We remark that, in what follows, $\partial \ff(S)$ stands for the shadow of $\ff(S)$. One of the key ingredients in \cite{F4} was
the following lemma.
\begin{lem}[\cite{F4}] If $\ff\subset {[n]\choose k}$ is initial and $\nu(\ff)\le s$, then \begin{equation}\label{eqshift2}
                        s|\partial \ff(\emptyset)|\ge |\ff(\emptyset)|.
                      \end{equation}
\end{lem}
Let us prove the following simple proposition, which, together with the lemma above, motivates the studies in the next subsection.
\begin{prop}\label{prop21} If $\ff$ is initial and $\nu(\ff)\le s$ then
  \begin{equation}\label{eqshadowmatch}
    \nu(\partial\ff(\emptyset))\le s.
  \end{equation}
\end{prop}
\begin{proof}
  Assume the contrary and let $G_1,\ldots, G_{s+1}\in \partial\ff(\emptyset)$ be pairwise disjoint. Now $G_i\in\partial\ff(\emptyset)$ implies that $G_i\cup \{x_i\}\in \ff(\emptyset)$ for some $x_i>s+1$. Then $G_i\cup \{i\}\prec G_i\cup \{x_i\}$ implies that the former is in $\ff$. Thus, we found $s+1$ pairwise disjoint sets $G_i\cup \{i\}$, $1\le i\le s+1$, all belonging to $\ff$, a contradiction.
\end{proof}

\subsection{Shadows of families satisfying $\pmb{\nu(\partial\G)\le s}$.}\label{secbeta1}
In this subsection, we work with an initial family $\G\subset{[m]\choose k}$ which satisfies $\nu(\partial \G)\le s$.
We start with the following corollary of Proposition~\ref{prop44}.
\begin{cor}\label{cor22} For every initial $\G\subset {[m]\choose k}$ such that $\nu(\partial\G)\le s$ and every $F\in \G$ there exists some $i$, $1\le i<k$, such that
   \begin{equation}\label{eqshift4}
    |F\cap [i(s+1)-1]|\ge i+1.
  \end{equation}
\end{cor}
\begin{proof}
  Remark that, provided that $j$ is the smallest element of $F$, we have $F':=F\setminus \{j\}\in \partial\ff$ and $|F'\cap [i(s+1)-1]|=\min\big\{\big|F\cap [i(s+1)-1]\big|-1, 0\big\}$. Then apply \eqref{eqshift3} to $F'$.
\end{proof}

For a set $F$, let us denote $i_F$ the largest $i$ for which \eqref{eqshift4} holds.
The  ideas below come the paper \cite{F5} due to the first author.

For each $F\in \G$ define the {\it tail} of $F$: $T(F):= F\setminus[i_F(s+1)-1]$. Let us split $\G:=\bigsqcup_{i=1}^{k-1}\G_i$, where $\G_i$ is the subfamily of all sets $F$ satisfying $i_F=i$.
Let us define the {\it restricted shadow} $$\partial_{res}\G:=\big\{F'\in \partial\G : \exists F\in \G \text{ s.t. } T(F)\subset F'\subset F\big\}.$$
\begin{lem}\label{lemgoodshadow} If $\G\subset {[m]\choose k}$ satisfies $\nu(\partial \G)\le s$ then \begin{equation}\label{eqgoodshadow}|\partial_{res}\G_i|\ge \frac{i+1}{is}|\G_i|.\end{equation}
Moreover, $\partial_{res}\G_i$ are disjoint for different $i$.
\end{lem}
Note that the coefficient $\frac{i+1}{is}$ in \eqref{eqgoodshadow} is at least $\frac{k}{(k-1)s}$, which is greater than $\frac 1s$. Using that $\G = \bigsqcup_{i=1}^{k-1}\G_i$ and the fact that $\partial_{res}\G_i$ are disjoint for different $i$, we obtain that $|\partial_{res}\G|\ge \frac{k}{(k-1)s}|\G|$. This improves the bound \eqref{eqshift2} by a little bit. To prove Theorem~\ref{thm1}, we need to analyze \eqref{eqgoodshadow} more carefully. In particular, this involves estimating the sizes of $\G_i$ (see the appendix).
 \begin{proof}[Proof of Lemma~\ref{lemgoodshadow}] Let us partition $\G:=\bigsqcup_{T\subset [m]}\G_+[T]$, where $$\G_+[T]:=\{F\in\G: T(F)=T\}.$$
 It is clear that $\G_+[T]$ indeed form a partition of $\G$.
 As earlier, we write $\partial_{res}\G_+[T]$ instead of $\partial_{res}(\G_+[T])$ for shorthand, and similarly in other situations. The next claim implies that $\partial_{res}\G_i$ are disjoint for different $i$.
 \begin{cla} We have
\begin{equation}\label{eq17}\partial_{res}\G_+[T]\cap \partial_{res}\G_+[T']=\emptyset\ \ \ \text{for} \ \ \ T\ne T'.\end{equation}
\end{cla}
\begin{proof} Take $F\in\G_+[T]$, $F'\in \G_+[T']$. The equation \eqref{eq17} is trivial if $|T|=|T'|$. Suppose that $|T'|<|T|$ and $A\in \partial_{res}\G_+[T]\cap \partial_{res}\G_+[T']$.  That is, there exist $F,F'\in\G$ such that $F\cap F'=A$, $T(F)=T$ and  $T(F')=T'$, $T,T' \subset A$. Recall that, by the definition of restricted shadow, $|A| = |F|-1 = |F'|-1$. Set $\{x\}:=F\setminus A,\ \{x'\}:=F'\setminus A$. Clearly, $x\ne x'$,  $x\in F\setminus T$ and $x' \in F' \setminus T'$. Recall that $T, T'$ consist of the last $|T|$ and $|T'|$ elements of $F$ and $F'$, respectively.

Consider an element $x''\in T'$. Note that $x''\ne x'$ and thus $x''\in A\subset F$. By definition, $x''$ is one of the $|T'|$ last elements of $F'$ and is thus one of the $|T|$ last elements of $F$. Therefore, $x''\in T$. We conclude that $T'\subsetneq T$. Together with $T\subset F'$ it implies that $i_{F'}> i_{F}$. Consequently, $F\setminus [i_{F'}(s+1)-1]\subset F\setminus [i_F(s+1)-1] = T\subset A\subset F'$. Therefore, $T'\subset F\setminus [i_{F'}(s+1)-1] \subset F'\setminus [i_{F'}(s+1)-1] = T'$, and we conclude that $F$ satisfies \eqref{eqshift3} for $i_{F'}$, a contradiction with maximality of $i_F$.\end{proof}


Define $\G_-[T]:=\big\{F\setminus T:F\in \G_+[T]\big\}$. Note that for any $T$ of size $k-i-1$ we have $i_F= i$ for any $F:$ $T(F) = T$, and thus $F\setminus T\subset [i(s+1)-1]$. Consequently, $\G_-[T]\subset {[i(s+1)-1]\choose i+1}$ and $\partial\G_-[T]\subset {[i(s+1)-1]\choose i}$. Clearly, $|\G_-[T]|=|\G_+[T]|$ and $|\partial\G_-[T]|=|\partial_{res}\G_+[T]|$.  Consider the bipartite graph between ${[i(s+1)-1]\choose i+1}$ and ${[i(s+1)-1]\choose i}$ with edges connecting pairs of sets, one of which contains the other. Via simple double counting, it is easy to see that, for any $T$ of size $k-i-1$,
\begin{eqnarray*}\notag \frac{|\partial\G_-[T]|}{{i(s+1)-1\choose i}}&\ge& \frac{|\G_-[T]|}{{i(s+1)-1\choose i+1}},\ \ \  \ \ \ \ \text{and hence}\\
|\partial_{res}\G_+[T]|&\ge& \frac{i+1}{is}|\G_+[T]|.\end{eqnarray*}
(We could have replaced $is$ by $is-1$, but this does not matter for us since we only study the case of large $s$.) Summing over all $T$ of size $k-i-1$ concludes the proof of the lemma.
\end{proof}

\section{Intersection of subsets and cliques in Kneser graphs is concentrated}\label{secknez}
Fix integers $m,l,t$, such that $m\ge tl$. Let $\G\subset {[m]\choose l}$ be a family and set  $\alpha:= |\G|/{m\choose l}$. Let $\eta$ be the random variable $|\G\cap \bb|$, where $\bb$ is chosen uniformly at random out of all $t$-matchings of $l$-sets. (More precisely, $\bb$ is taken uniformly at random from the set of all $t$-tuples of pairwise disjoint $\ell$-element sets.) Clearly, we have \begin{equation}\label{eq3}\E[\eta] = \alpha t.\end{equation}

Using the eigenvalue properties of Kneser graphs, we deduce that $\eta$ is concentrated around its mean. We present two bounds, the first one based on Chebyshev's inequality, and the second one based on the Azuma--Hoeffding inequality (\cite{Hoeff}, \cite{Azu67}).

We recall that the Kneser graph $KG_{m,l}$ is the graph on the vertex set ${[m]\choose l}$ and with the edge set formed by pairs of disjoint sets. For a family $\G\subset{[m]\choose l}$, we denote by $e(\G)$ the number of edges of $KG_{m,l}$ induced between the members of $\G$.

\begin{lem}\label{lemcheb}Suppose that $m,l,t$ are integers and $m\ge tl$. Let $\G\subset {[m]\choose l}$ be a family, and $\alpha:=|\G|/{m\choose l}$. Let $\eta$ be the random variable equal to the size of the intersection of $\G$ with a $t$-matching $\bb$ of $l$-sets,  chosen uniformly at random. Then $\E[\eta]=\alpha t$ and, for any positive $\beta$, we have \begin{equation} \Pr[|\eta-\alpha t|\ge \beta t]\le \frac{2\alpha(1-\alpha)}{\beta^2 t}.\end{equation}
\end{lem}
\begin{proof}
 We have $\eta=\eta_1+\ldots+\eta_t$, where $\eta_i$ is the indicator function of the event $A_i$ that the $i$-th set in $\bb$ belongs to $\G$. Then $\Var[\eta] = \sum_{i} \Var [\eta_i]+\sum_{i\ne j}\Cov[\eta_i, \eta_j].$ It is easy to see that, since $\Pr(A_i)=\alpha$, we have $\Var[\eta_i] = \alpha(1-\alpha)$. The covariance of $\eta_i$ and $\eta_j$ for $i\ne j$ is estimated in the following proposition.

\begin{prop}\label{prop4}
For any $i,j\in [t]$, where $i\ne j$, we have $|Cov[\eta_i, \eta_j]|\le \frac{\alpha(1-\alpha)}{t-1}$.
\end{prop}
\begin{proof}
  We have $$\Cov[\eta_i,\eta_j] = \Pr(A_i\cap A_j)-\Pr(A_i)\Pr(A_j) = \Pr(A_i\cap A_j)- \alpha^2.$$
  At the same time, $\Pr(A_i\cap A_j)$ is equal to the probability that a randomly chosen edge in the Kneser graph $KG_{m,l}$ has both ends in $\G$ (Here we use that every pair of disjoint sets is contained in the same number of $t$-matchings.) Let $M:={m\choose l}$ be the number of vertices in $KG_{m,l}$ and $D := {m-l\choose l}$ be the degree of a vertex in $KG_{m,l}$. Due to regularity, the largest eigenvalue of (the adjacency matrix of) $KG_{m,l}$ is equal to $D$. Let $\lambda$ be the second-largest absolute value of an eigenvalue  of $KG_{m,l}$. It is known (see, e.g., the celebrated paper \cite{L} due to Lov\'asz) that $\lambda = {m-l-1\choose l-1}$, and thus $\frac {\lambda}{D} = \frac{l}{m-l} \le \frac 1{t-1}$. Using the result of Alon and Chung \cite{AC}, we get that the number of edges of $KG_{m,l}$ induced in $\G$ satisfies
  \begin{equation}\label{eqalch}\big|e(\G)- \frac {\alpha^2}2DM\big|\le \frac 12 \lambda \alpha(1-\alpha) M.\end{equation}
 On the other hand, we  have $\Pr(A_i\cap A_j) = \frac{e(\G)}{DM/2}$ and thus \begin{equation}\label{eqjoint}\big|\Pr(A_i\cap A_j)-\alpha^2|\le \frac{\lambda\alpha(1-\alpha)}{D}.\end{equation}
 Substituting $\lambda/D \le 1/(t-1)$, we get
 $$\big|\Cov[\eta_i,\eta_j]\big| =\big|\Pr(A_i\cap A_j)- \alpha^2\big|\le \frac{\alpha(1-\alpha)}{t-1}.$$
\end{proof}

We conclude that \begin{equation}\label{eq4}\Var[\eta] \le \alpha(1-\alpha)\Big(t+t(t-1)\cdot \frac 1{t-1}\Big) = 2t\alpha(1-\alpha).\end{equation}

Using Chebyshev's inequality, \eqref{eq3} and \eqref{eq4}, we conclude that, for any positive $\beta$, we have
\begin{equation*} \Pr[|\eta-\alpha t|\ge \beta t]\le \frac{\Var[\eta]}{\beta^2t^2} \le \frac{2\alpha(1-\alpha)}{\beta^2 t}.\end{equation*}
\end{proof}
\vskip+0.2cm

We can get a much stronger concentration result if we use martingales.
\begin{thm}\label{thmazuma}
  In the setting of Lemma~\ref{lemcheb}, we have
\begin{align}\label{eq8} \Pr\big[|\eta - \alpha t|\ge 2\beta \sqrt t\big]&\le 2 e^{-\beta^2/2}.
\end{align}
\end{thm}

\begin{proof} Let $X_0,\ldots, X_t$ be the following exposure martingale: $$X_i:=\E[\eta\mid\eta_1,\ldots, \eta_{i-1}].$$
In particular, $X_0 = \E[ \eta]$ and $X_t = \eta$. Let us show that $|X_i-X_{i-1}|\le 2$. We actually show that something slightly stronger holds. Assume that the choice of the first $i-1$ sets $B_1,\ldots, B_{i-1}$ in the random matching $\bb$ is fixed (and thus the choice of $\eta_1,\ldots, \eta_{i-1}$ is also fixed). We will show that
\begin{equation}\label{eq7}\big|\E[\eta\mid \eta_i,B_{i-1},\ldots, B_1]-E[\eta\mid B_{i-1},\ldots, B_1]\big|\le 2\end{equation}
for any choice of $B_1,\ldots, B_{i-1}$, where, somewhat unconventionally, conditioning on $B_j$ denotes the event that the $j$-th set of $\bb$ is $B_j$. This implies $|X_i-X_{i-1}|\le 2$. Indeed, fix a choice of $\eta_1,\ldots, \eta_{i-1}$. This choice induces a probability subspace of all choices of $\bb$, which have the fixed values of $\eta_1,\ldots, \eta_{i-1}$. This probability space can be further refined by specifying $B_1,\ldots, B_{i-1}$. Let us denote by $\mathcal B_i$ the set of all choices of ${\bf B}:=(B_1,\ldots, B_{i-1})$ which comply with the fixed choice of $\eta_1,\ldots, \eta_{i-1}$. Then we have
$$\sum_{\mathbf{B}\in \mathcal B_i} \Pr\big[{\bf B}\mid\eta_{i-1},\ldots, \eta_{1}\big]\Big(\E\big[\eta\mid\eta_i,\mathbf{B}\big]-\E\big[\eta\mid\mathbf{B}\big]\Big) = \E\big[\eta\mid\eta_i,\ldots, \eta_1\big]-\E\big[\eta\mid\eta_{i-1},\ldots, \eta_1\big].$$
Thus, if the expression in brackets on the left hand side has absolute value at most 2 (which is exactly what \eqref{eq7} states), then the right hand side has absolute value at most $2$, which is in turn equivalent to $|X_i-X_{i-1}|\le 2$.

Next, we prove \eqref{eq7}. Fix $B_1,\ldots, B_{i-1}$ and consider the Kneser graph $KG_{m',l}$ on $Y':=[m]\setminus \bigcup_{j=1}^{i-1}B_j$. Here, $|Y'|=m' \ge t'l$ with $t':=t-i+1$. Put $\G':= \G\cap {Y'\choose l}$ and $\alpha':=|\G'|/{m'\choose l}$. Put $\eta'=\eta'_i+\ldots+\eta'_t$ to be equal to the intersection of $\G'$ with a randomly chosen $t'$-matching $\bb'$ of $l$-sets in $Y'$, where $\eta'_j$ and the corresponding even $A'_j$ are defined analogously to $\eta_j,\ A_j$.

Then, clearly, $\E[\eta'\mid\eta'_i] = \E[\eta\mid\eta_i,B_{i-1},\ldots, B_1]$ and $\E[\eta'] = \E[\eta\mid B_{i-1},\ldots, B_1]$. We have $E[\eta']=\alpha't'$. At the same time, $$\E[\eta'\mid \eta'_i] =\eta'_i+ \sum_{j=i+1}^t \E[\eta'_j\mid \eta'_i] = \eta'_i+(t'-1)\E[\eta'_t\mid \eta'_i].$$
In order to prove $|X_i-X_{i-1}|\le 2$, we need to show that for both $\eta'_i=0$ and $\eta'_i=1$ the value of the last expression is between $\alpha't'-2$ and $\alpha't'+2$. Let us first consider the case $\eta'_i=1$.
$$\E[\eta'_t\mid \eta'_i=1] = \E[\eta'_t\mid A'_i] = \Pr[A'_t\mid A'_i]=\frac{\Pr[A'_t\cap A'_i]}{\Pr[A'_i]}.$$
Using \eqref{eqjoint}, we conclude that the following holds.
\begin{equation}\label{eqmart}\big|\E[\eta'_t\mid A'_i]-\alpha'\big |\le \frac{\lambda'(1-\alpha')}{D'}\le \frac{(1-\alpha')}{t'-1},\end{equation}
where $D'$ and $\lambda'$ are the degree and the second-largest absolute value of an eigenvalue of $KG_{m',l}$, respectively. Indeed, as before, we have $\frac{\lambda'}{D'} \le \frac 1{t'-1}$. Therefore, we conclude that in the case $\eta'_i=1$
$$\big|\E[\eta'\mid\eta'_i=1]-\alpha't'\big| \le 1-\alpha'+\big|(t'-1)\E[\eta'_t\mid A'_i]-\alpha'(t'-1)\big|\overset{\eqref{eqmart}}{\le} 2-2\alpha'\le 2.$$
Similarly, we can obtain
$$\E[\eta'_t\mid \eta'_i=0] = \frac{\Pr[A'_t]-\Pr[A'_t\cap A'_i]}{1-\Pr[A'_i]} = \frac{\alpha'-\Pr[A'_t\cap A'_i]}{1-\alpha'}.$$
Using \eqref{eqjoint}, we get that
$$\big|\E[\eta'_t\mid\eta'_i=0]-\alpha'\big|\le \frac{\lambda'\alpha'}{D'} \le \frac{\alpha'}{t'-1}$$ and, doing the same calculations as before, we infer that
$$|\E[\eta'\mid\eta'_i=0]-\alpha't'|\le 2\alpha'\le 2.$$

Thus, we can apply the Azuma--Hoeffding inequality to $X_0,\ldots, X_t$ and conclude that \eqref{eq8} holds. Note that we have $2\beta$ instead of $\beta$ in \eqref{eq8} due to the fact that $X_i$ are $2$-Lipschitz.
\end{proof}

For technical reasons, in case when $\alpha$ is small, we will need to compare the probability that $X_t$ got ``very big'' and the probability that it got ``just big''. Below we give a proposition that formalises this.

\begin{prop}\label{propazumastupid}In the notations above, assume that $\E[\eta]=\alpha t$. Fix a constant $0<C\le 1/2$ such that $C\ge \alpha$ and $ C^2t\ge 16$. Then \begin{equation}\label{eq10}\Pr[\eta\ge 4C t]\le 2e^{-\frac{C^2t}{2}}\Pr\big[|\eta-2Ct|\le Ct\big].\end{equation}
\end{prop}
\begin{proof} Let $\rho$ be the random variable, which is equal to $i$ in case $i$ is the first step at which $X_i\ge 2Ct-2$. If there is no such step, then put $\rho := -1$. Note that, since $X_0,\ldots, X_t$ form a $2$-Lipschitz martingale, we have $2Ct-2\le X_i< 2Ct$ for $i$ defined as above. Moreover, if the value of $X_t$ exceeds $4Ct$, then $2Ct-2\ge Ct\ge\alpha t$ and the value of $X_i$ must become bigger than $2Ct-2$ at some step and so $\rho$ is assigned an integer from $0$ to $t$.
\begin{equation}\label{eq9}\Pr[X_t\ge 4C t]=\sum_{i=0}^t\Pr[X_t\ge 4Ct\mid \rho=i]\Pr[\rho=i].\end{equation}
Let us bound the following related quantity:
$$\Pr\big[|X_t-2Ct+1|\le Ct-1\mid\rho = i\big].$$ The sequence $X_i,\ldots, X_t$, conditioned on $\rho = i$, is a 2-Lipschitz martingale with the expected value lying between $2Ct-2$ and $2Ct$. Therefore, we can apply the Azuma--Hoeffding inequality to this martingale and conclude that
\begin{eqnarray*}\Pr\big[X_t-(2Ct-1)\ge 1+2\beta \sqrt {t-i}\mid\rho=i\big]&\le &  e^{-\beta^2/2},\\
\Pr\big[X_t-(2Ct-1)\le -1-2\beta \sqrt {t-i}\mid\rho=i\big]&\le & e^{-\beta^2/2}.\end{eqnarray*}
Using these inequalities with $t-i$ bounded from above by $t$ and $\beta=Ct^{1/2}$ and  $\frac{Ct-2}{2t^{1/2}}$ for the numerator and denominator, respectively, we get
$$\frac{\Pr[X_t\ge 4Ct\mid \rho=i]}{\Pr[|X_t-2Ct|\le Ct\mid\rho = i]}\le\frac{\Pr[X_t\ge 4Ct\mid \rho=i]}{\Pr[|X_t-2Ct-1|\le Ct-1\mid\rho = i]}\le  \frac{e^{-\frac{C^2t}{2}}}{1-2e^{-\frac{t(C-2/t)^2}{8}}}\le 2e^{-\frac{C^2t}{2}},$$
where the last inequality holds since $t(C-2/t)^2\ge \frac 34 C^2t\ge 12$ and $2e^{-12/8}\le 1/2$.

 Using this bound and the fact that $\eta = X_t$, we can continue \eqref{eq9} as follows:
\begin{equation*}\Pr[\eta\ge 4C t]\le 2e^{-\frac{C^2t}{2}}\sum_{i=0}^t\Pr\big[|\eta-2Ct|\le Ct\mid\rho=i\big]\Pr[\rho=i] \le 2e^{-\frac{C^2t}{2}}\Pr\big[|\eta-2Ct|\le Ct\big].\end{equation*}
\end{proof}

\section{Proof of Theorem~\ref{thm1}}\label{sec2}

For convenience, we prove Theorem~\ref{thm1} in the following, slightly modified form.
\begin{thm}\label{thm1v1} For any $\eps>0$ there exists $s_0=s_0(\eps)$, such that for any $s\ge s_0$ and $n\ge s+(1.666+\eps)s(k-1)$ the conclusion of Theorem~\ref{thm1} is valid.
\end{thm}
Theorem~\ref{thm1v1} easily implies Theorem~\ref{thm1}. Indeed, we only have to choose $\epsilon<5/3-1.666$ and apply Theorem~\ref{thm1v1}. The rest of the section is concerned with the proof of Theorem~\ref{thm1v1}.

We prove Theorem~\ref{thm1v1} by induction on $k$. The case $k=3$ is verified by the first author in \cite{F6}. Using shiftedness, it is easy to obtain the formula $$m(n,k,s)\le m(n-1,k,s)+m(n-1,k-1,s),$$ valid for any $n\ge (s+1)k$. At the same time, we have $${n\choose k}-{n-s\choose k}= {n-1\choose k}-{n-s-1\choose k}+{n-1\choose k-1}-{n-s-1\choose k-1}.$$
(It is important to note that in the above recursions $n$ and $k$ change but $s$ is fixed. This is essential because we only prove Theorems~\ref{thm1} and~\ref{thm1v1} for $s>s_0$.)
Consequently, if we proved the EMC for $k$-uniform families and $n=s+(1.666+\eps)s(k-1)$, then, using the inductive hypothesis for $(k-1)$-uniform families and the formulas above, we can conclude that the EMC is valid for any $n\ge s+(1.666+\eps)s(k-1)$. Therefore, we only need to prove the EMC for $n=s+(1.666+\eps)s(k-1)$. (Note that we omit integer parts when they are unimportant.)

Recall that families $\ff_1,\ldots, \ff_{s+1}$ are {\it cross-dependent}, if there are no $F_1\in\ff_1,\ldots, F_{s+1}\in \ff_{s+1}$ such that $F_1,\ldots, F_{s+1}$ are pairwise disjoint. We say that $\ff_1,\ldots, \ff_{s+1}$ are {\it nested} if $\ff_1\supset \ff_2\supset\ldots\supset\ff_{s+1}$.
The following somewhat technical lemma is the key ingredient in the proof. It resembles
\cite[Theorem 3.1]{F4}, but the analysis is more complicated. Theorem~\ref{thmazuma} plays a crucial role in the proof, providing us with much more control over the situation than a simpler averaging argument used in \cite{F4}.

\begin{lem}\label{lem2} For any $\epsilon>0$ there exists $s_0\in \mathbb N$, such that the following holds for any $s\ge s_0$. Let $\ff_1,\ldots, \ff_{s+1}\subset {Y\choose l}$ be cross-dependent and nested, and suppose that $|Y|\ge tl$ for some $t\in \mathbb N$. If for some $x,q$ with $x,q\le s+1$ and $\alpha\in[0,1]$ we have $|\ff_{s+1}|=\alpha {|Y|\choose l}$,  $(\alpha+\epsilon) t\le \frac{sx}q$, and $t\ge s+x+1$, then
\begin{equation}\label{eqlem1} |\ff_1|+|\ff_2|+\ldots+|\ff_s|+q|\ff_{s+1}|\le s{|Y|\choose l}.\end{equation}
\end{lem}
We defer its proof to the next subsection and first finish the proof of Theorem~\ref{thm1v1}.
Recall that, for a subset $S\subset [s+1]$, we use the  notation $\ff(S):=\{F-S:F\in\ff, F\cap [s+1]=S\}$.
The next lemma translates the statement of Lemma~\ref{lem2} to our situation.
\begin{lem}\label{lemtobeta} Fix $c=1.666$. For any $\eps>0$ there exists $s_0\in \mathbb N$, such that the following holds for any $s\ge s_0$. Put $n= s+(c+\eps)s(k-1)$ and consider an initial family $\ff\subset{[n]\choose k}$ satisfying $\nu(\ff)\le s$.   Assume that $q'|\partial\ff(\emptyset)|=|\ff(\emptyset)|$, and let $\lambda, \beta>0$ be such that $|\partial\ff(\emptyset)|= \lambda{n-s-1\choose k-1}$ and  $|\ff(\emptyset)|= \beta{n-s-1\choose k}$. Then the EMC is true, provided that at least one of the following inequalities holds:
\begin{align}\label{cond1} \lambda\le& \frac{s(c-1)}{q'c}\ \ \ \  \text{or}\\
\label{cond2} \beta\le& \frac{(c-1)k}{c^2(k-1)}-\eps.
\end{align}
\end{lem}
 We could have provided a more concise statement, since giving bounds on both $\lambda$ and $\beta$ in the statement is redundant (the parameters are interconnected via $q'$). However, this form of the statement illuminates the actual logic of the proof.

\begin{proof} To prove the lemma, it is sufficient to show that
\begin{equation}\label{eq27}|\ff(\{1\})|+\ldots+|\ff(\{s+1\})| +|\ff(\emptyset)|\le s{n-s-1\choose k-1}\end{equation}
holds. Indeed, for any subset $S\subset [s+1]$, $|S|\ge 2$, we have $\aaa(S)\supset \ff(S)$, since the former contains all possible such sets. The inequality \eqref{eq27}, in turn, gives that $$\sum_{S\subset[s+1]: |S|\in\{0,1\}}|\aaa(S)|\ge \sum_{S\subset[s+1]: |S|\in\{0,1\}}|\ff(S)|.$$
In total, this gives $|\aaa|\ge |\ff|$.

Thus, our main task is to verify \eqref{eq27}. Fix some $\eps'$, $0<\eps'<\eps$, which choice would be clear later.  Assume that $|\ff(s+1)|= \alpha{n-s-1\choose k-1}$. As $\ff(\emptyset)\subset \ff$, we clearly have $\nu(\ff(\emptyset))\le s$, and thus, using \eqref{eqshift2}, we conclude that $q'\le s$. We apply Lemma~\ref{lem2} with $\epsilon:=\eps'$, $l:=k-1$ and $q:=1+q'\frac \lambda\alpha$ to $\ff(\{1\}),\ldots,\ff(\{s+1\})$. Since $\ff$ is initial, these families are nested. Also, these families are cross-dependent: otherwise, if $F_i\in \ff(\{i\})$ are pairwise disjoint, then $F_i\cup\{i\}$ form an $(s+1)$-matching in $\ff$.

Moreover, we have  $q|\ff(\{s+1\})|= |\ff(\{s+1\})|+q'\frac \lambda\alpha|\ff(\{s+1\})| = |\ff(\{s+1\})|+q'|\partial\ff(\emptyset)|= |\ff(\{s+1\})|+|\ff(\emptyset)|$. Therefore, inequality  \eqref{eqlem1} with these parameters implies \eqref{eq27}.

To see that \eqref{eqlem1} holds, we need to verify that the assumptions on $\alpha, t,x$ from Lemma~\ref{lem2} are satisfied. Put $Y:=[s+2,n]$. We have $|Y|=n-s -1= (c+\eps)s(k-1)-1$, so we may put $t:=(c+\eps)s-1$, $x:=(c-1+\eps)s-2$ in order to satisfy the inequality $t\ge x+s+1$. Thus, we are left to verify that for some positive $\eps'$
$$q(\alpha+\eps')\big((c+\eps)s-1\big)\le s\big((c-1+\eps)s-2).$$
It is easy to see that one can find positive $\eps'=\eps'(\eps)$ and a sufficiently large $s_0$ so that the above holds for $s\ge s_0$, provided that
\begin{equation*}(q-1)\alpha cs\le s^2(c-1)\end{equation*} holds. Note that $(q-1)\alpha = q'\lambda$, and thus the displayed inequality is equivalent to \eqref{cond1}. (Note that we simply discarded constants and epsilons and replaced $q$ with $q-1$, which is possible since the inequality above is non-trivial only if $q>(c-1)s/c$ and thus $(q-1)/q$ can be made as close to $1$ as is necessary.)
On the other hand, by the definition of $\lambda, \beta,q'$ we have $\lambda q' = \frac{|\ff(\emptyset)|}{{n-s-1\choose k-1}}$, and thus $\beta = \frac{|\ff(\emptyset)|}{{n-s-1\choose k}}= \lambda q'\frac{k}{n-s-k}\ge \frac{\lambda q'k}{(c+\varepsilon)s(k-1)}$. Thus, the inequality \eqref{cond1} is implied by
\begin{equation*} (c+\eps)s\beta\frac{k-1}k\le s\frac{c-1}c \ \ \ \ \ \ \Leftarrow \ \ \ \ \beta\le \frac{(c-1)k}{c^2(k-1)}-\eps.\end{equation*}
(Note that we use $(c+\eps)/c\le (\beta+\eps)/\beta$.) The last condition is exactly \eqref{cond2}.\end{proof}

To complete the proof of Theorem~\ref{thm1v1}, we need to find good bounds on either $\alpha$ or $\beta$. This is done in the following lemma, whose proof is deferred to the appendix.

\begin{lem}\label{lembeta}
  For $c=1.666$ either \eqref{cond1} or \eqref{cond2} is valid.
\end{lem}
This lemma, combined with Lemma~\ref{lemtobeta}, concludes the proof of Theorem~\ref{thm1v1}.
Unfortunately, the proof of Lemma~\ref{lembeta} involves some quite technical parts, in particular, obtaining good bounds on some expressions involving sums and products of binomial coefficients. At the heart of it, however, we find ideas from Section~\ref{secbeta1}, combined with induction. Roughly speaking, if $q'$ from Lemma~\ref{lemtobeta} is significantly smaller than $s$, then we use \eqref{cond1}, combined with the fact that we assume by induction that Theorem~\ref{thm1} is valid for $k-1$ (and thus we can get good bounds on $\alpha$ using \eqref{eqemc}). If $q'$ is large, then, using the ideas from Section~\ref{secbeta1}, we are able to say something about the structure of $\ff(\emptyset)$ and, most importantly, get good upper bounds on $\beta$, concluding via \eqref{cond2}. Thus, the conceptual part of the proof, based on \cite{F5}, is presented in Section~\ref{secbeta1}, while the necessary tedious estimates are deferred to the appendix.





\subsection{Proof of Lemma~\ref{lem2}. An auxiliary lemma}


Choose uniformly at random $t$ pairwise disjoint sets $B_1,\ldots, B_t\in {Y\choose l}$ and define $\mathcal B:=\{B_1,\ldots, B_t\}$. The expected size of $\mathcal B\cap \ff_i$ is $t|\ff_i|/{|Y|\choose l}$.
\begin{lem}\label{lem3} Let $1\le x,q\le s+1$ and $t\ge s+x+1$, $t\in \mathbb N$. Let the families $\ff_1,\ldots, \ff_{s+1}\subset {Y\choose l}$ be cross-dependent and nested, and suppose that $|Y|\ge tl$. Fix any $t$-matching $\mathcal B$ of $l$-element sets in $Y$. We have
\begin{equation}\label{eq1}|\bb\cap \ff_1|+\ldots+|\bb\cap \ff_s|+q|\bb\cap \ff_{s+1}|\le st+q|\bb\cap \ff_{s+1}|-sx\end{equation}
for $|\bb\cap \ff_{s+1}|\ge x$ and
\begin{equation}\label{eq2}|\bb\cap \ff_1|+\ldots+|\bb\cap \ff_s|+q|\bb\cap \ff_{s+1}|\le st-|\bb\cap \ff_{s+1}|\Big(x-\frac{q|\bb\cap \ff_{s+1}|}{s+1}\Big)\end{equation}
for $|\bb\cap \ff_{s+1}|\le x$.
\end{lem}

\begin{proof} Consider the bipartite graph between $\bb$ on the one side and $\ff_i$, $i=1,\ldots, s+1$, on the other side, with edges connecting $B_j$ and $\ff_i$ if and only if $B_j\in \ff_i$. Put weight $1$ on the edges incident to $\ff_1,\ldots, \ff_s$ and weight $q$ on the edges incident to $\ff_{s+1}$. This graph has no matching of size $s+1$ (otherwise, the families are not cross-dependent), therefore, all edges can be covered by $s$ vertices. Note that each neighbor of $\ff_{s+1}$ has degree $s+1$ and therefore must be included in the vertex cover (and thus $\ff_{s+1}$ is not in the vertex cover).  Assume that $q_1$ vertices are chosen among $\ff_1,\ldots, \ff_s$, $q_2:=|\bb\cap\ff_{s+1}|$ vertices are chosen among the neighbors of $\ff_{s+1}$, and $q_3$ vertices are chosen among other vertices of $\bb$. Note that  $q_1+q_2+q_3=s$ (we may assume that the equality holds by adding extra vertices if needed). Then the total weight of all the edges in the graph is at most (see the explanations on Fig.~\ref{fig1})
\begin{multline}\label{edgecalc}q_1t+q_2(q+s-q_1)+q_3(s-q_1)= q_1t+(q_2+q_3)\Big(q_2+q_3+\frac{qq_2}{q_2+q_3}\Big)\le \\
st-(q_2+q_3)\Big(s+x+1-\Big(q_2+q_3+ \frac{qq_2}{q_2+q_3}\Big)\Big).\end{multline}
\begin{figure}\centering
\begin{tikzpicture}[scale=1]


\foreach \x in {1,2,3} \foreach \y in{-1,1,2,4,5,6,7,10}
{\draw[thick,red] (\x,0)-- (\y,3);
}
\foreach \x in {4,5,6,7} \foreach \y in{4,7,10}
{\draw[very thick,dotted,black] (\x,0)-- (\y,3);
}
\foreach \x in {4,5,6}
{\draw[thick,dashed,blue] (\x,0)-- (2,3);
}

\foreach \y in{4,7,10}
\draw[ultra thick,black] (8,0)-- (\y,3);

\node at (1,-0.5) { $\ff_{1}$};
\node at (2,-0.5) { $\ff_{2}$};
\node at (7,-0.5) { $\ff_{s}$};
\node at (8.2,-0.5) { $\ff_{s+1}$};
\node at (-1,3.5) { $B_{1}$};
\node at (0,3.5) { $B_{2}$};
\node at (10,3.5) { $B_{t}$};
\node at (4,3.5) { $\ldots$};
\node at (4,-0.5) { $\ldots$};
\draw[red,thick] (2,0) ellipse (1.5 and 0.5);
\foreach \x in {-1,...,10}
\filldraw ({\x},3) circle (3pt);
\foreach \x in {1,...,8}
\filldraw ({\x},0) circle (3pt);
\foreach \x in {4,7,10}
\node at (\x,3) [minimum size=0.5cm,draw] (\x) {};
\node at (2,3) [minimum size=0.5cm,circle,draw]  {};

\end{tikzpicture}
\caption{Illustration to the count of weights of edges in the bipartite graph. The edges are represented by segments (solid, dashed or dotted).  Thick black edges (and only them) have weight $q$. The vertices inside the ellipse ($q_1$), squares ($q_2$) and circle  ($q_3$) form an edge cover.  There are at most $q_1t$ edges incident to the vertices inside the ellipse (in red). The total weight of the thick edges and the dotted edges is at most $q_2(q+s-q_1)$. The total weight of the dashed edges is $q_3(s-q_1)$.} \label{fig1}
\end{figure}
Let us analyze the contribution of the last term
\begin{equation}\label{eqmainevil} (q_2+q_3)(s+x+1-(q_2+q_3))-qq_2.\end{equation}
The bigger the expression is, the smaller the right hand side  in~\eqref{edgecalc} is. For any  $q_2\le x$ the first summand is at least $q_2(s+x+1-q_2)$.\footnote{Indeed, the expression of the form $z(C-z)$ for fixed positive $C>0$ and $z\in S\subset [0,C]$ is minimized when $\min\{z,C-z\}$ is the smallest. In our case, the second multiple is at least $s+x+1-s\ge x+1>q_2$ in our assumption. Moreover, obviously, $q_2+q_3\ge q_2$.} Thus the expression \eqref{eqmainevil} is at least $$q_2(s+x+1-q_2-q)\ge q_2\Big(\frac x{s+1}(s+1-q)+\frac q{s+1}(x-q_2)\Big)=q_2\Big(x-\frac{qq_2}{s+1}\Big),$$
where the last expression is exactly as stated in \eqref{eq2}. Assume that  $x\le q_2\le q_2+q_3\le s$ (the second inequality holds by the definition). Then the first summand in \eqref{eqmainevil} is at least its value for $q_2+q_3=s+1$, and we get that \eqref{eqmainevil} is at least
\begin{equation}\label{eqnonoptimal} (s+1)x-qq_2\ge sx-qq_2,\end{equation}
as \eqref{eq1} states. This concludes the proof of Lemma~\ref{lem3}.
\end{proof}

In the next section, we combine Lemma~\ref{lem3} with the findings from Section~\ref{secknez} to conclude the proof of Lemma~\ref{lem2}.

\subsection{Completing the proof of Lemma~\ref{lem2}}
Let us show that, averaging over the choice of $\bb$, we have
\begin{equation}\label{eq11}\E\big[|\bb\cap \ff_1|+\ldots+|\bb\cap \ff_s|+q|\bb\cap \ff_{s+1}|- st\big]\le 0.\end{equation}
Essentially, it just follows from the concentration of the intersection $|\bb\cap \ff_{s+1}|=:q_2$ and the fact that on average it contributes negative terms due to Lemma~\ref{lem3} and the hypothesis of Lemma~\ref{lem2}.

Due to the condition on $\alpha$ in Lemma~\ref{lem2}, the average value of $q_2$ in Lemma~\ref{lem3} is $\alpha t\le \frac{sx}q-\epsilon t$. Assume first that, say, $\alpha\ge 3\epsilon$. Then, applying Theorem~\ref{thmazuma}, we get that the probability that $|\bb\cap \ff_{s+1}|\ge \frac {sx}q$  is at most $2e^{-\epsilon^2t/8}=o(t^{-4})$ for any sufficiently large $t$. We may trivially bound the contribution of each of the terms with $|\bb\cap \ff_{s+1}|\ge \frac {sx}q$ as $q|\bb\cap \ff_{s+1}|-sx\le t^2$. Thus, the (positive) contribution of these terms to the expectation in \eqref{eq11} is  $o(t^{-2})$. On the other hand, using Theorem~\ref{thmazuma} again, we see that the value of $q_2$ falls in the interval $[(\alpha-\epsilon/2)t,(\alpha+\epsilon/2)t]$ with probability at least $1/2$. Each of these terms, according to Lemma~\ref{lem3}, make a negative contribution of at least
$$\min\Big\{\frac{\epsilon}2t,(\alpha-
\frac{\epsilon}2)t\frac{\epsilon}2t\Big\}\ge \epsilon^2t$$
absolute value to the expectation, where the first term comes from \eqref{eq1} and the second comes from \eqref{eq2}. Clearly, given that $t$ is large enough, the negative contribution of these terms to the expectation is at least $\frac 12 \epsilon^2t=\Omega(t^{-2})$, which completes the proof of \eqref{eq11} in the case when $\alpha\ge 3\epsilon$.

The case $\alpha<3\epsilon$ is done analogously, with Theorem~\ref{thmazuma} replaced by Proposition~\ref{propazumastupid}. Unfortunately, we need  this technical twist since the contribution of the terms with $q_2=0$ to the expectation is $0$, and we need to use this tool to formally express that we cannot be in a situation when $|\bb\cap \ff_{s+1}|$ takes value $0$ with probability close to $1$, and some large value with probability close to $0$.\\

Once we have \eqref{eq11}, it is easy to finish the proof of Lemma~\ref{lem2}. Indeed, we have
$$st\ge \E[|\bb\cap \ff_1|+\ldots+|\bb\cap \ff_s|+q|\bb\cap \ff_{s+1}|]= t\big(|\ff_1|+\ldots+|\ff_s|+q|\ff_{s+1}|\big)/{|Y|\choose l}.$$

\section{Proof of Theorem~\ref{thmemcgen}}\label{secgen}
The statement of Theorem~\ref{thmemcgen} follows from Theorem~\ref{thm1} and the following proposition, which allows to extend any bounds on the EMC to the full range.

\begin{lem}\label{prop66}
  Assume that for some $0<\alpha<1$, $n_0,k$ and $0<x<1/k$ we have $m(n,k,xn)/{n\choose k}\le \alpha$ for all $n\ge n_0$. Then for any $y$ satisfying $x\le y\le 1/k$ there exists $n_1$ such that for any $n\ge n_1$ with $yn\in \mathbb N$ we have $m(n,k,yn)/{n\choose k}\le \alpha+\frac{y-x}{1/k-x}(1-\alpha)$.
\end{lem}
We remark that the upper bound on $m(n,k,yn)$ is a convex combination of the assumed bound on $m(n,k,xn)$ and the trivial bound $m(n,k,n/k)/{n\choose k}\le 1$. (Note that $yn = xn+ \frac{y-x}{1/k-x}(n/k-xn)$.)
\begin{proof}
Let us prove the statement by induction on $yn$. It clearly holds for $y=x$. Put $t:=n/k$.  Our main tool is inequality \eqref{eqalch}, applied to the Kneser graph $KG_{n,k}$ and its subgraphs. Take a family $\ff\subset{[n]\choose k}$ with $\nu(\ff)=yn$ and $|\ff|=m(n,k,yn)$. Put $\beta:=|\ff|/{n\choose k}$, $D:={n-k\choose k}$ and, for a set $A\in {[n]\choose k}$, denote by $d_{\ff}(A)$ the number of sets from $\ff$, disjoint with $A$. Note that, for a randomly chosen $A$, $\mathrm E[d_{\G}(A)/D]=\beta$.
  Inequality \eqref{eqalch} for $KG_{n,k}$ implies that
  $$\big|\sum_{A\in\ff}d_{\ff}(A)-\beta D|\ff|\big|\le \lambda (1-\beta)|\ff|,$$
  where $\lambda$ is as in the proof of Proposition~\ref{prop4}. Recall that $\lambda/D=1/(t-1)$.
  That is, if we fix a random set $A$ from $\ff$ and consider a subfamily $\ff_A:=\{F\in \ff: F\cap A=\emptyset\}$, then in expectation $\big||\ff_A|/{n-k\choose k}-\beta\big|\le \frac{1-\beta}{t-1}$. Concluding, we get that on average
  \begin{equation}\label{eqloss}|\ff_A|/{n-k\choose k} \ge \beta-\frac{1-\beta}{t-1}.\end{equation}
Take a set $A\in \G$ satisfying the inequality above. Remark that $yn-1=\big(y-\frac{\frac 1k-y}{t-1}\big)(n-k)$. We have $\nu(\ff_A)\le yn-1$ and thus
$$|\ff_A|/{n-k\choose k}\le m\big(n-k,k, \big(y-\frac{1/k-y}{t-1}\big)(n-k)\big)/{n-k\choose k}\le \alpha+\frac{1-\alpha}{1/k-x}\Big(y-\frac{1/k-y}{t-1}-x\Big)$$
by the induction hypothesis. Combining \eqref{eqloss} and the inequality displayed above, one concludes that $\beta\le \alpha+\frac{1-\alpha}{1/k-x}(y-x)$, as stated.
\end{proof}

To deduce Theorem~\ref{thmemcgen}, we just note that $m(n,k,(\frac 53k-\frac 23)^{-1}n)\le {n  \choose k}-{n-s\choose k}$ for sufficiently large $n$ due to Theorem~\ref{thm1}, and thus due to Lemma~\ref{prop66}
$$m(n,k,(\gamma k-(\gamma-1))^{-1}n)\le {n\choose k}-{n-s\choose k}+\frac{(\gamma k-(\gamma-1))^{-1}-(\frac 53k-\frac 23)^{-1}}{k^{-1}-(\frac 53k-\frac 23)^{-1}}{n-s\choose k}.$$
Simplifying the expression above, we get that
$$m(n,k,(\gamma k-(\gamma-1))^{-1}n)\le {n\choose k}-\frac{\gamma-1}2\cdot\frac{5k-2}{\gamma k-(\gamma-1)}{n-s\choose k},$$ as stated.

\section{Applications of Theorem~\ref{thm1} and related questions}\label{secappl}
In Section~\ref{sec11} we discuss relaxations of the EMC, in Section~\ref{sec12} we deduce corollaries for Dirac thresholds, in Section~\ref{sec13} we briefly mention other combinatorial applications of the EMC, in Section~\ref{sec14} we speak about the relation of the EMC to the problems concerning deviations of sums of random variables.
\subsection{Relaxations of the Erd\H os Matching Conjecture}\label{sec11}
A {\it fractional matching} in $\ff\subset {[n]\choose k}$ is a weight function $w:\ff\to [0,1]$, such that $\sum_{F\in \ff: i\in F}w(F)\le 1$ for every $i\in [n]$. It is a relaxation of the notion of  (integer) matching, for which we are only allowed to have $w(F)\in\{0,1\}$ for every $F\in \ff$. The {\it size} of a fractional matching is $\sum_{F\in\ff}w(F)$. Let us denote by $\nu^*(\ff)$ the size of the largest fractional matching in $\ff$. In particular, we have $\nu^*(\ff)\ge \nu(\ff)$. Note also that $\nu^*(\aaa(n,k,s))<s+1$ and $\nu^*(\aaa_0(k,s))<s+1$. For an integer $s$, let $m^*(n,k,s)$ be the maximum number of edges in a family $\ff\subset{[n]\choose k}$ such that $\nu^*(\ff)<s+1$. The following natural relaxation of the Erd\H os Mathing Conjecture was proposed in Alon et. al. \cite{aletal}.
\begin{gypoerd}[fractional version, \cite{aletal}]
  We have \begin{equation}\label{emc2} m^*(n,k,s)=\max\Big\{{n\choose k}-{n-s\choose k}, {k(s+1)-1\choose k}\Big\}.
          \end{equation}
\end{gypoerd}
One may be tempted to state the conjecture above for non-integral values of $s$, but the situation in that case is more complicated, see \cite{aletal}.

An interesting relaxation of the conjecture concerns the regime when $k$ is fixed, $s$ is linear in $n$ and $n\to \infty$. It is more convenient to change the parametrisation and assume that $\nu(\ff)=xn$ for some fixed $x\le 1/k$. For such $x$, it is not difficult to see that one has
$$\lim_{n\to\infty} |\aaa(n,k,xn)|/{n\choose k}=1-(1-x)^k\ \ \ \ \text{ and }\ \ \ \ \lim_{n\to\infty} |\aaa_0(k,xn)|/{n\choose k}=(kx)^k.$$

The following two conjectures are natural relaxations of the two versions of the EMC presented above.
\begin{gypoerd}[asymptotic version, \cite{aletal}] For any fixed $k\ge 2$ and positive $x\le 1/k$ one has \begin{equation}\label{emc3}
                          \lim_{n\to\infty} m(n,k,xn)/{n\choose k} = \max\big\{1-(1-x)^k, (kx)^k\big\}.
                        \end{equation}
\end{gypoerd}

\begin{gypoerd}[asymptotic fractional version, \cite{aletal}] For any fixed $k\ge 2$ and positive $x\le 1/k$ one has \begin{equation}\label{emc4}
                          \lim_{n\to\infty} m^*(n,k,xn)/{n\choose k} = \max\big\{1-(1-x)^k, (kx)^k\big\}.
                        \end{equation}
\end{gypoerd}
In all the variants of the EMC that we stated, the lower bound is obviously attained. Since it is only the upper bound that is interesting, the last conjecture is clearly the weakest out of all four.

\subsection{Dirac thresholds}\label{sec12} An active area of research in extremal combinatorics stems from the famous Dirac's criterion for Hamiltonicity: any $n$-vertex graph with minimum degree at least $n/2$ contains a Hamilton cycle. For $0\le d\le k-1$ and $\ff\subset {[n]\choose k}$, let us denote by $\delta_d(\ff)$ the minimal $d$-degree of $\ff$, that is, $\delta_d(\ff):=\min_{S\in {[n]\choose d}}\big|\big\{F\in \ff: S\subset F\big\}\big|.$ Let us give the following general definitions.
\begin{eqnarray}\label{dirac1}
  m_d(n,k,s)&:=&\max\Big\{\delta_d(\ff): \ff\subset {[n]\choose k}, \nu(\ff)< s+1\Big\}\ \ \ \ \text{and} \\
  \label{dirac2} m_d^*(n,k,s)&:=&\max\Big\{\delta_d(\ff): \ff\subset {[n]\choose k}, \nu^*(\ff)< s+1\Big\}.
\end{eqnarray}
(Note that, to comply with the EMC, the definitions are slightly different from the ones normally used in the literature.) In particular, if we substitute $d=0$ in the definitions above, then we get back to the functions $m(n,k,s)$ and $m^*(n,k,s)$, while if we substitute $s=n/k$ (given that $k$ divides $n$, which we assume tacitly), then the functions $m_d(n,k,n/k-1)$, $m_d^*(n,k,n/k-1)$ provide us with sufficient conditions for the existence of perfect (fractional) matchings. Let us denote these two functions $m_d(n,k)$, $m^*_d(n,k)$ for shorthand. There is extensive literature on the subject, and we refer the reader to the survey \cite{Zh16}. Let us summarize some of the known results. The problem of determining $m_1(n,k,s)$ was considered in \cite{BDE} and \cite{DH81}. Some of the first results on the topic were due to R\"odl, Ruci\'nski and Szemer\'edi \cite{RRS06,RRS09}: they determined the exact values of $m^*_{k-1}(n,k)$ and $m_{k-1}(n,k)$, respectively. The first one is roughly $n/k$, while the second one is roughly $n/2$. The reason for such a difference in behaviour is the so-called ``divisibility barrier'' for the existence of perfect integral matchings. The values of $m_d(n,k)$ were determined asymptotically for $d\ge 0.42 k$ (\cite{RRS09, Pikh, Han16}). The values of $m_d^*(n,k)$ were determined exactly for $d\ge k/2$. Basically, all known asymptotical results for $m_d(n,k)$, $m_d^*(n,k)$ follow from the aforementioned result \cite{F4} of the first author on the EMC via the following considerations, presented in \cite{aletal}. First, we state without proof the following proposition, which is a straightforward generalization of \cite[Proposition 1.1]{aletal}.

\begin{prop}\label{prop1.1}
  We have $m_d^*(n,k,s)\le m^*(n-d,k-d,s)$.
\end{prop}

Next, the following general theorem was proven in \cite{aletal} in the asymptotic form and  refined in \cite{TZ16} to give the exact part.
\begin{thm}[\cite{aletal}, \cite{TZ16}]\label{thmequiv}
  Fix $k,d\in \mathbb N$ with $1\le d\le k-1$. If $\limsup_{n\to\infty} m^*_d(n,k)/{n-d\choose k-d}=c^*$ for some $c^*\in (0,1)$, then
  \begin{equation}\label{fracint1}
    \limsup_{n\to\infty} m_d(n,k)/{n-d\choose k-d}=\max\{c^*,1/2\}.
  \end{equation}
Moreover, if $c^*<1/2$ then there exists $n_0$, such that $m_d(n,k)$ is determined exactly for all $n\ge n_0$.
\end{thm}

The counterpart of this result for smaller matchings (at least for $d\ge 1$) was proven by K\"uhn, Osthus and Townsend \cite{KOT}.

\begin{thm}[\cite{KOT}]\label{thmequiv2}
  Fix $k,d\in \mathbb N$ with $0\le d\le k-1$ and $x<1/k$. Then
  \begin{equation}\label{fracint1}
    \limsup_{n\to\infty} m^*_d(n,k,xn)/{n-d\choose k-d}=\limsup_{n\to\infty} m_d(n,k,xn)/{n-d\choose k-d}.
  \end{equation}
\end{thm}
The authors of \cite{KOT} proved Theorem~\ref{thmequiv2} using the Weak Hypergraph Regularity Lemma~\cite{FR92}. In the paper, we will give a proof of Theorem~\ref{thmequiv2} based on an extension of the approach from \cite{aletal}, which is hopefully simpler and may have an interest of its own. See Section~\ref{secnibble} for details. We note that one direct consequence of Theorem~\ref{thmequiv2} is that the fractional and integral asymptotic forms of the EMC are equivalent for any $k$ and $x<1/k$.

The following is an immediate corollary of Theorem~\ref{thm1} combined with Proposition~\ref{prop1.1}
\begin{cor}\label{cor11} There exist $s_0$, such that the following holds for any $s\ge s_0$. For any $n,k,s,d$ satisfying $1\le d\le k-1,$ $s\le n/k$ and $n-d\ge \frac 53 (k-d)s-\frac 23 s$ we have $m^*_d(n,k,s) = {n-d\choose k-d}-{n-s-d\choose k-d}.$
\end{cor}
In particular, this gives exact values of $m^*_d(n,k)$ for all sufficiently large $n$ and $5d\ge 2k-2$.\footnote{For the case $5d=2k-2$ one has to use Theorem~\ref{thm1v1}, which says that the conclusion of Theorem~\ref{thm1} is valid if one replaces $5/3, 2/3$ by $5/3-10^{-4}, 2/3-10^{-4}$.} This includes such new cases as $(k,d)=(8,3),\ (9,4),\ (10,4), (11,4)$ etc.
 Exact values were previously known only for $(k,d)=(3,1)$, $(4,1)$ (cf. \cite{Khan,KOTr}) and in the range $d\ge k/2$, while asymptotic solutions were also given in \cite{aletal} for pairs satisfying $k-4\le d\le k-1$.

Using Theorem~\ref{thmequiv2}, we obtain the following asymptotic result.
\begin{cor} Fix $k\in \mathbb N$ and some positive $x<1/k$. Then for any $d$ satisfying $1\le d\le k-1$ and $\frac 53 (k-d)x-\frac 23 x<1$ we have $\limsup_{n\to\infty} m_d(n,k,xn)/{n-d\choose k-d} = 1- (1-x)^{k-d}.$
\end{cor}

We can say slightly more about $m_d(n,k)$.
\begin{prop}\label{prop33}
  We have $\limsup_{n\to\infty} m_d^*(n,k)/{n-d\choose k-d} \le 1-\frac{\gamma-1}2\cdot\frac{5k-2}{\gamma k-(\gamma-1)}(1-1/k)^{k-d}$ provided $d> \frac{\gamma-1}{\gamma}(k-1)$. If additionally $\frac{\gamma-1}2\cdot\frac{5k-2}{\gamma k-(\gamma-1)}(1-1/k)^{k-d}\ge \frac 12$ then \begin{equation}\label{eqdirac}\limsup_{n\to\infty} m_d(n,k)/{n-d\choose k-d}=\frac 12.\end{equation} Moreover, we know the exact value of $m_d(n,k)$ for all $n\ge n_0$. In particular, \eqref{eqdirac} holds for all $d\ge 3k/8$.
\end{prop}

\begin{proof}[Proof of Proposition~\ref{prop33}] The first part of the statement follows from Theorem~\ref{thmemcgen} and Proposition~\ref{prop1.1}, using that $n-d=sk-d\ge \gamma(k-d)s+(\gamma-1)s$ is equivalent to $d> \frac {\gamma-1}{\gamma}(k-1)$ in the limit $s\to\infty$. (Also recall that ${n-s-d\choose k-d}/{n-d\choose k-d}\to (1-1/k)^{k-d}$ as $n\to \infty$.) The second part directly follows from the first part of the proposition and Theorem~\ref{thmequiv}. The final conclusion that \eqref{eqdirac} holds for $d\ge 3k/8$ follows from the fact that for $d= 3k/8$ we may find $\gamma$ such that both $d> \frac{\gamma-1}{\gamma}(k-1)$ and $(\gamma-1)\frac{5k-2}{\gamma k-(\gamma-1)}(1-1/k)^{k-d}> 1$ hold. This can be verified by a computer-aided computation.
\end{proof}

We note that, although the determination of $m_d(n,k)$ asymptotically reduces to the corresponding fractional problem, other methods are needed to determine $m_d(n,k)$ {\it exactly} (see. e.g., \cite{HPS, TZ13, TZ16}). The main technique used for this group of problems is {\it absorption}. On a very high level, one searches for small subfamilies in the original family, which, once an almost-spanning matching is found, can be used to cover {\it any} small remainder by a perfect matching. This is a very powerful technique, which allows to find much more general structures. One remarkable example of the use of absorption is the second proof of the existence of combinatorial designs given by Glock, Lo, K\"uhn and Osthus \cite{GLKO}. (This result was first proved by Keevash \cite{Kee} using other methods.)

\subsection{Other combinatorial applications}\label{sec13}
There are several other problems in which the EMC plays an important role. In particular, results on fractional version of the EMC were used by Alon, Huang and Sudakov \cite{AHS} to prove the Manickam-Mikl\'os-Singhi conjecture for $n\ge 33k^2$. They also note that, as was pointed out by Ruci\'nski, the Manickam-Mikl\'os-Singhi conjecture is actually equivalent to a variant of fractional version of the EMC.

We say that families $\ff_1,\ldots, \ff_{s+1}$ are {\it cross-dependent}, if there are no $F_i\in\ff_i$, $i=1,\ldots, s+1$, such that $F_1,\ldots, F_{s+1}$ are pairwise disjoint. In \cite{HLS}, one of the main ingredients of the proof of the EMC for $n\ge 3k^2s$ was the result stating that if for some $n\ge (s+1)k$ the families $\ff_1,\ldots, \ff_{s+1}\subset{[n]\choose k}$ are cross-dependent, then $\min_{i}|\ff_i|\le s{n-1\choose k-1}$. They asked whether an analogue of \eqref{eqemc0} always holds for cross-dependent families. They could prove it for $n\ge 3k^2s$. Keller and Lifshitz \cite{KLchv} proved it for $n\ge f(s) k$ with some $f(s)$. Unfortunately, the proof of the first author \cite{F4}, as well as the proof of the present result, breaks for cross-dependent families. We have recently showed this for $n>12ks\log(e^2s)$ \cite{FK17}, and it was announced by
Keevash, Lifshitz, Long, and Minzer that this holds for $n>Csk$ with some large $C$ as a consequence of general sharp threshold-type results. Several questions in this spirit were independently asked by Aharoni and Howard \cite{AhH}.\vskip+0.1cm

{\bf Remark. } Since the appearance of the first version of this paper, the paper of Keevash, Lifshitz, Long and Minzer has appeared \cite{KLLM}. The second author of this paper has also managed to extend the ideas of the present proof to the rainbow EMC, proving it for $n>3esk$ and $s>10^7$. Together with the result of \cite{FK17}, this leads to the resolution of the rainbow EMC for all $s$ and $n>300sk$.\vskip+0.1cm


The EMC was used in the study of the non-uniform analogue of the EMC due to Erd\H os and Kleitman \cite{Kl}. For recent progress, see \cite{FK8}--\cite{FK7}, and especially \cite{FK8} for the connection between the uniform and the non-uniform problems.

Among other applications of the EMC, let us point out that the EMC was used in \cite{OY} and \cite{FK7} to obtain progress in the following question: what is the maximum number of (non-empty) colors one can use in the coloring of ${[n]\choose k}$ without forming  an $(s+1)$-matching of sets of pairwise distinct colors. In other words, what is the maximum $t$, such that ${[n]\choose k}=\ff_1\sqcup \ldots \sqcup \ff_t$, where all $\ff_i$ are non-empty and any $s+1$ of them are cross-dependent?

\subsection{Deviations of sums of nonnegative variables}\label{sec14}
 Assume that $X_1,\ldots, X_k$ are nonnegative independent, identically distributed random variables with mean $x<1/k$. Put $\mathbf X:=(X_1,\ldots, X_k)$. Put
$$p_k(x):=\sup_{\mathbf X}\Pr[X_1+\ldots +X_k\ge 1].$$
The value of $p_2(x)$ was determined by Hoeffding and Shrikhande \cite{HS}.  \L uczak, Mieczkowska and \v Sileikis~\cite{LMS} proposed the following conjecture, which states that for every positive $k$ and $0\le x\le 1/k$ we have
\begin{equation}\label{conjsam}
 p_k(x)=\max\{1-(1-x)^k, (kx)^k\}.
          \end{equation}
We note that it is easy to see that $p_k(x)=1$ for $x\ge 1/k$. The authors of \cite{LMS} proved the equivalence of \eqref{conjsam} and \eqref{emc4}, which implied that \eqref{conjsam} is true for $k=3$ and any $x$, as well as for any $k$ and $x\le \frac 1{2k-1}$.  Theorem~\ref{thm1}, combined with the aforementioned equivalence, immediately implies the following corollary.
\begin{cor}
  The equality \eqref{conjsam} holds for any $x\le \frac 3{5k-2}$.
\end{cor}

Actually, the idea to relate \eqref{emc4} and conjectures similar to \eqref{conjsam} appeared already in \cite{aletal}, \cite{AHS}. The following general conjecture was stated by Samuels \cite{Sam66}. Let $X_1,\ldots, X_k$ be independent random variables with means $x_1\le \ldots \le x_k$. Assume that $\sum_{i=1}^k x_i<1$ and let
$$p(x_1,\ldots,x_k):= \sup_{X_1,\ldots,X_k}\Pr[X_1+\ldots +X_k\ge 1].$$
Put $\bar p_k(x):=p(x,\ldots,x)$. Note that the difference between $p_k(x)$ and $\bar p_k(x)$ is that in the latter we do not require the random variables to be identically distributed. Thus, clearly, $\bar p_k(x)\ge p_k(x)$. Samuels \cite{Sam66} conjectured that for all admissible $x_1,\ldots, x_k$
\begin{equation}\label{conjsam2}
  p(x_1,\ldots,x_k)=\max_{t=0,\ldots,k-1}1-\prod_{i=t+1}^k\Big(1-\frac{x_i}{1-\sum_{j=1}^tx_j}\Big).
\end{equation}
It is not difficult to come up with the example of random variables that show the ``$\ge$''-part of \eqref{conjsam2}. Moreover, as it is shown in \cite{aletal}, for $x_1=\ldots = x_k=:x$ with $0\le x\le 1/(k+1)$, the maximum of the right hand side is attained for $t=0$, which suggests the following conjecture: for any $0\le x\le 1/(k+1)$, we have
\begin{equation}\label{conjsam3}
  \bar p_k(x)=1-(1-x)^k.
\end{equation}
Samuels \cite{Sam66,Sam68} verified \eqref{conjsam2} for $k\le 4$, which means that \eqref{conjsam3} and \eqref{conjsam} are valid for $k\le 4$, $x\le 1/(k+1)$. Combined with the equivalence of \eqref{conjsam} and Theorem~\ref{thmequiv2}, we get the following corollary.
\begin{cor}
  The equalities \eqref{emc3} and \eqref{emc4} hold for $k=4$, $x\le 1/5$.
\end{cor}
Comparing with our main result, a similar corollary of Theorem~\ref{thm1} (more precisely, its ``optimized'' version, Theorem~\ref{thm1v2}) would imply \eqref{emc3} and \eqref{emc4} (and also \eqref{eqemc0} and \eqref{emc2}) for $x\le 0.18$.

The case $x=1/(k+\delta)$, where $\delta>0$ is meant to be a small constant, of \eqref{conjsam3} was studied by Feige \cite{Fe06} in the context of some algorithmic applications, in particular, estimating the average degree of a graph. He managed to prove the following bound:
\begin{equation}\label{eqfe}
  \bar p_k\big(\frac 1{k+\delta}\big)\le \frac {12}{13} \ \ \ \ \ \text{for any }\delta\ge \frac 1{12}.
\end{equation}
In particular, this bound, together with the aforementioned equivalences, implies that for any $\eps>0$ there exist $s_0$, such that for all $s\ge s_0$ and $\delta\ge 1/12$ we have
\begin{equation}\label{eqemcgen2}
  m\big(n,k,\frac n{k+\delta}\big)\le \Big(\frac {12}{13}+\eps\Big){n\choose k}.
\end{equation}
This is much stronger than the bound \eqref{eqemcbad} for large $k$. (The latter implies $m(n,k,\frac n{k+\delta})\le \frac {k}{k+\delta}{n\choose k}$.) Later, the bound \eqref{eqfe} for $\delta =1$ was improved in \cite{HZZ} to $\bar p_k\big(\frac 1{k+1}\big)\le \frac 78$.

\section{Proof of Theorem~\ref{thmequiv2}}\label{secnibble}
Our proof of this theorem follows the same steps as the proof of Theorem~\ref{thmequiv}. Fix some small $\eps>0$. Put $c^*_{x+\eps}:=\limsup_{n\to \infty} m^*_d(n,k,(x+\eps)n)/{n-d\choose k-d}$. Assume that $n_0=n_0(\eps)$ is large enough and take $n\ge n_0$. Consider a family $\ff\subset {[n]\choose k}$ satisfying $\delta_d(\ff)>(c^*_{x+\eps}+\eps){n-d\choose k-d}$. We use the following claim, used in \cite{aletal} to prove Theorem~\ref{thmequiv}. Note that the $o(1)$-notation is with respect to $n\to\infty$.

\begin{cla}[\cite{aletal}]
  There exist sets $R^i\subset [n]$, $i=1,\ldots, n^{1.1}$, such that the families $\ff_i:=\{F\in\ff: F\subset R^i\}$ satisfy the following conditions.
  \begin{itemize}
    \item[(i)] For every $v\in [n]$, the number $Y_v$ of sets $R^i$ containing $v$ is $(1+o(1))n^{0.2}$,
    \item[(ii)] every pair $u,v\in[n]$ is contained in at most two sets $R^i$,
    \item[(iii)] every set $F\in \ff$ is contained in at most one set $R^i$,
    \item[(iv)] for all $i=1,\ldots, n^{1.1}$, we have $|R^i|=(1+o(1))n^{0.1}$ and
    \item[(v)] for all $i=1,\ldots, n^{1.1}$ we have $\delta_d(\ff_i)\ge (c^*_{x+\eps}+\eps/2){|R^i|-d\choose k-d}$.
  \end{itemize}
\end{cla}

Next, still following \cite{aletal}, we use (v) and find fractional matchings $w^i:\ff_i\to [0,1]$ of size at least $(x+\eps)|R^i|$ for $\ff_i$, $i=1,\ldots, n^{1.1}$. We construct a random family $\h$ by including  $F\in\ff_i$ with probability $w^i(F)$. (Note that this procedure is well-defined due to (iii).) The family $\h$ with high probability has the following properties.
\begin{itemize}
  \item[\textbf A] $\Delta(\h)\le(1+o(1)) n^{0.2}.$
  \item[\textbf B] the number of edges containing any two given vertices is at most $n^{0.1}$.
  \item[\textbf C] the average degree of a vertex in $\h$ is at least $(x+\eps+o(1))k n^{0.2}.$
\end{itemize}

The verification of \textbf A and \textbf B is done as in \cite{aletal}, while \textbf C is easy to obtain, since the expected number of edges in $\h$ is $\sum_{i=1}^{n^{1.1}}\sum_{F\in \ff_i}w^i(F)\ge n^{1.1}\cdot (x+\eps)|R^i|\ge (x+\eps+o(1))x n^{1.2},$ and it is highly concentrated around the mean (easily verified via Chernoff-type bounds).

The only twist we have to add to the proof of the authors of \cite{aletal} is the following useful generalization of the theorem due to Frankl and R\" odl \cite{FR85} and Pippenger and Spencer \cite{PS89}. In what follows,  a {\it $k$-uniform hypergraph} is used in almost the same sense as a family $\ff\subset{[n]\choose k}$, with the only difference that a hypergraph may have multiple edges. A {\it codegree} of two vertices in a hypergraph is the number of edges containing both of them (counted with multiplicities). Let $d(\h)$ be the average degree of $\h$.

\begin{thm}\label{thmnewnibble}
  For any $k\in \mathbb N$ and $\delta>0$ there exists $n_0$ and $\eps>0$, such that the following holds for any $n\ge n_0$. Let $\h$ be an $n$-vertex $k$-uniform hypergraph. If $d(\h)/\Delta(\h)=c>0$ and the codegree of any two vertices $v,w$ in $\h$ is at most $\eps \Delta(\h)$, then there exists a matching $\mathcal M\subset \h$ covering at least a $c-\delta$-proportion of vertices of $[n]$.
\end{thm}
We prove this theorem in the next subsection, using the method from \cite[Theorem~2.13]{KKKO}. We note that, using the same argument, one may prove an obvious common generalization of Theorem~\ref{thmnewnibble} and \cite[Theorem~2.13]{KKKO}, however, this is not needed for our purposes. Let us now finish the proof of Theorem~\ref{thmequiv2}.

Using Theorem~\ref{thmnewnibble}, we conclude that $\h$ contains a matching covering at least an $xk$-proportion of vertices, provided $n$ is sufficiently large. In other words, $c_x:=\limsup_{n\to\infty} m_d(n,k,xn)/{n-d\choose k-d}\le c_{x+\eps}^*+\eps$. On the other hand, of course, $c_x\ge c^*_x$. We can make $\eps$ arbitrarily small, and thus we conclude that $c_x=c^*_x$, provided that $c^*_x$ is continuous as a function of $x$, $x\in (0,1/k)$. This is proven in the next lemma.

\begin{lem}
  The function $c_x^*$ is monotone and continuous as a function of $x$, where $x\in [0,1/k)$.
\end{lem}
We remark that the same proof would work for $x=1/k$ and $d=0$.
In this respect, $c_x^*$ and $c_x$ behave differently, since $c_x$ is not continuous at $1/k$ (due to the parity-based constructions, see the discussion in the introduction).
\begin{proof}
  The monotonicity is obvious. Fix some $x\in (0,1/k)$. We show that, for any $\eps>0$, there exists $\delta$, such that $c^*_{x-\delta}\ge c^*-\eps k$. To do so, it is clearly sufficient to show that, given a family $\ff\subset {[n]\choose k}$ satisfying $xn-1\le \nu^*(\ff)\le xn$, we can obtain a family $\ff'\subset {n\choose k}$, such that $\nu^*(\ff')\le (x-\delta)n$ and $\delta_d(\ff')\ge \delta_d(\ff')-\eps{n-d\choose k-d}$. Let us take sufficiently large $n$ depending on $x,\eps, \delta$. For simplicity, we assume that $xn$ is an integer.

  Consider $\ff$ as above and, using LP-duality, consider the function $w:[n]\to [0,1]$, such that $\sum_{i\in [n]}w(i)\le xn$ and $w(F):=\sum_{i\in F}w(i)\ge 1$ for each $F\in\ff$. Without loss of generality, assume that $w(1)\ge w(2)\ge\ldots\ge w(n)$. We may also assume that $\ff = \big\{F\in {[n]\choose k}:w(F)\ge 1\big\}$.  Let $s+1$ be the smallest index $i$ such that $w(i)<1/k$. Clearly, $s\ge xn-1$. Indeed, if it is not true, then any set in $\ff$ intersects the  first $xn-2$ elements. Thus, the  function assigning weight $1$ to each of the elements $1,\ldots, xn-2$ and $0$ to others would be a fractional covering function for $\ff$, contradicting the assumption $\nu^*(\ff)\ge xn-1$. If $d\ge 1$, then we do the following preprocessing. Replace each of $w(n-d+1),\ldots, w(n)$ with $\phi:=\sum_{i=n-d+1}^n w(i)/d$ and redefine $\ff$ with respect to these weights. If $d=0$ then put $\phi:=0$. It is easy to see that this operation does not decrease the minimal $d$-degree. Let us also note that $\phi\le x <1/k$.

  Define the weight function $w'(i)$ by putting $w'(i):=w(i)$ for $i\in [(1-\eps)s]\cup [s+1,n]$ and $w'(i):=\phi$ for $i\in [(1-\eps)s+1,s]$. Define $\ff':=\{F\in {[n]\choose k}:w'(F)\ge 1\}$. Note that $\nu^*(\ff')\le w'([n])\le xn-\eps(1/k-\phi)s\le (x-\delta)n$ for $\delta:= \eps x(1/k-\phi)/2$. Moreover, the minimal $d$-degree of $\ff'$ is still achieved on the $d$-subset $[n-d+1,n]$, and is at least $\delta_d(\ff)-\eps s{n-d-1\choose k-d-1}\ge \delta_d(\ff)-\eps k{n-d\choose k-d}$. (We simply used the fact that the number of sets containing both $[n-d+1,n]$ and some $i\in [(1-\eps)s+1,s]$ is at most $\eps s{n-d-1\choose k-d-1}$.) This concludes the proof.
\end{proof}

\subsection{Proof of Theorem~\ref{thmnewnibble}}
Let $\delta(\h)$ stand for the minimal degree of $\h$. We use the following theorem due to Pippenger and Spencer.
\begin{thm}[\cite{PS89}] \label{lem: Pippenger}
For any $k\in \mathbb N$ and $\delta>0$ there exists $n_0$ and $\eps>0$, such that the following holds for any $n\ge n_0$. Let $\h$ be an $n$-vertex $r$-uniform hypergraph satisfying $\delta(\h)\geq (1-\eps) \Delta(\h)$ and the codegree of any two vertices in $\h$ is at most $\eps \Delta(\h)$, then $\h$ can be partitioned into $(1+\delta)\Delta(\h)$ matchings.
\end{thm}

Consider $\h$ as in the statement of Theorem~\ref{thmnewnibble}. Given such $\h$, in the proof of \cite[Theorem~2.13]{KKKO} the authors construct the hypergraph $\G$ containing $\h$,
which additionally satisfies the following properties: $\Delta(\G)\le (1+\eps^{1/3})\Delta(\h),$ $\delta(\h')\ge (1-\eps)\Delta(\h)$ and the codegree of any two vertices is at most $\eps^{1/4}\Delta(\h)$. One can then apply Theorem~\ref{lem: Pippenger} to $\G$ with $\delta^2, \eps^{1/4}$ playing the roles of $\delta$ and $\eps$, respectively,
 and obtain $(1+\delta^2)\Delta(\G)\le (1+\delta/k)\Delta(\h)=:t$ matchings $\mathcal M_1,\ldots,\mathcal M_t$ partitioning the set of edges of $\G$. Since $\mathcal M_1,\ldots, \mathcal M_t$ cover $\h$, the
 expected number of $k$-sets in the intersection $\mathcal M\cap \h$ is $|\h|/t\ge \frac {cn}{k+\delta}\ge \frac {cn}k-\frac {\delta n}k$. Choose  $\mathcal M_i$ which has intersection of at least expected size.
  Then it covers at least  a $(c-\delta)$-proportion of vertices of $\h$.

\section{Concluding remarks}
The bounds we present in the paper can be further optimized. In particular, here is what we can get for $k< 10$ using the approach presented in the appendix.
\begin{thm}\label{thm1v2}
  The EMC is true for all $s\ge s_0$ and $n\ge c_k(k-1)s+s$, where $c_4=1.58$, $c_5=1.6$, $c_6,c_7=1.61$, $c_8,c_9=1.62$.
\end{thm}
Even that we extended the range for which the EMC is proved, we feel that new ideas are needed to prove the EMC for all $n>(s+1)k$. Answering the following question would be very helpful for some further progress on the EMC.

For a family $\ff\subset {[n]\choose k},$ $\nu(\ff)\le s$, $n\ge k(s+1)$, let us define its $s$-diversity $\gamma_s(\ff)$ by $$\gamma_s(\ff):=\min_{T\in {[n]\choose s}} |\ff(\bar T)|, \ \ \ \ \ \ \ \text{where}$$
$\ff(\bar T):=\{F\in\ff: F\cap T=\emptyset\}$. Note that for an initial family the minimum is attained for $T = [s]$, i.e., in our notation $\gamma_s(\ff)=|\ff(\{s+1\})|+|\ff(\emptyset)|$. Recall the definition \eqref{ainks} of $\aaa_i(n,k,s)$.
\begin{prb}
  For any $n\ge k(s+1)$, find the maximum of $\gamma_s(\ff)$ among (shifted) families $\ff\subset {[n]\choose k}$ with $\nu(\ff)\le s$. When does one have
  \begin{equation}\label{eqdiver}\gamma_s(\ff)\le \max_{i\in [2,k]}\gamma_s(\aaa_i(n,k,s))?\end{equation}
\end{prb}
Resolving this problem completely, or at least for shifted families, would provide much better bounds on $\beta$ from Lemma~\ref{lemtobeta}.
Also, it would immediately provide us with a good universal bound on the $m(n,k,s)$.
We could prove \eqref{eqdiver} only for $n\ge n_0(k,s)$. That proof along with some other diversity results appears in \cite{FK15}. We note that, although we believe that \eqref{eqdiver} always holds for shifted families, it is not always true for general families, even for $s=1$. See the paper \cite{Kup21} of the second author for details.

In \cite{FK18}, we studied a general problem that includes the EMC as a subcase. There are many interesting questions that arise there.

We note that it would be also very interesting to extend the stability result for the EMC proved in \cite{FK7} to the new range $n>\frac 53 sk$.

Finally, let us report on some applications of the concentration method developed in this paper that has appeared since the first version of this paper. In a recent paper \cite{Kup33}, the second author managed to extend the methods of this paper in order to prove the rainbow version of the EMC for $n>3esk$ and $s$ sufficiently large. An analogous rainbow question for $k$-partite hypergraphs was fully solved for $s>500$ by Kiselev and the second author \cite{KK}. Finally, Kiselev and the two authors \cite{FKK} used the concentration of intersection with a random matching in order to derive another concentration result that concerns the changes in the density of a uniform family when restricted to a random subset of fixed size. They applied it to advance in a question on the number of distinct intersections in an intersecting family.

\section{Acknowledgments}
We thank the referees for carefully reading the paper and pointing out several problems with the exposition. We also would like to thank Andrew Treglown for bringing some of the references to our attention.

\section{Appendix. How to get good bounds on $\beta$ or $\alpha$ in Lemma~\ref{lemtobeta}}
In this section we prove Lemma~\ref{lembeta}. We note that the numerical bounds we get are by no means optimal even within our approach. We had to make a compromise: on the one hand, to obtain better bounds on $c$, and thus on $n$, in the main Theorem, and, on the other hand, not to flood the paper with tedious estimates of expressions involving sums and products of binomial coefficients.

\subsection{Bounds on $|\G_i|$}
To use \eqref{eqgoodshadow} effectively, we need to get bounds on $|\G_i|$. We will use the bounds in this subsection to show that the size of $\G_i$ decreases exponentially as $i$ increases. Recall that $\G\subset {[m]\choose k}$. 

For each set $F\in \G_i$ we have $|F\cap[i(s+1)-1]|=i+1$ and $F\cap [i(s+1),(i+1)(s+1)-1]=\emptyset$, implying  \begin{footnotesize}\begin{equation}\label{eqest1}
\frac{|\G_i|}{{m\choose k}}\le \frac{{i(s+1)-1\choose i+1}{m+1-(i+1)(s+1)\choose k-i-1}}{{m\choose k}} = {k\choose i+1}\frac{\prod_{j=1}^{i+1}(i(s+1)-j) \prod_{j=1}^{k-i-1}(m+2-(i+1)(s+1)-j)}{\prod_{j=1}^{k}(m+1-j)}.
\end{equation}\end{footnotesize}
Let us do some auxiliary computations.
We prove the following useful inequality, valid for any
$a,b>0$ and $k_1,k_2\in \mathbb N,$ satisfying $k_1< a$ and $k_2<b$.
\begin{footnotesize}\begin{equation}\label{eqbincoeff}\frac{\prod_{i=1}^{k_1}(a-i)\prod_{i=1}^{k_2}(b-i)} {\prod_{i=1}^{k_1+k_2}(a+b-i)}=\frac{a^{k_1}b^{k_2}}{(a+b)^{k_1+k_2}}\cdot
\frac{\prod_{i=1}^{k_1}(a+b-\frac{a+b}ai)\prod_{i=1}^{k_2}(a+b-\frac{a+b}bi)} {\prod_{i=1}^{k_1+k_2}(a+b-i)}\le \frac{a^{k_1}b^{k_2}}{(a+b)^{k_1+k_2}}.\end{equation}\end{footnotesize}
To see that the last inequality holds first note that for any $j_1,j_2\in \mathbb N $ and $x>1$, if $j_1\frac{a+b}a\le x$ and $j_2\frac{a+b}b\le x$ then $x\ge j_1+j_2$. Indeed, taking a convex combination of the two inequalities assumed to be valid, we get $x\ge \frac{a}{a+b}\cdot j_1\frac {a+b}a+\frac{b}{a+b}\cdot j_2\frac{a+b}b=j_1+j_2$. Then, write $\frac {a+b}ai$, $i=1,\ldots,k_1$ and $\frac{a+b}bi$, $i=1,\ldots, k_2$, as one sequence $S$, in increasing order. Due to the claim above, the $j$-th member of this sequence is bigger than or equal to $j$. This implies that the before-last fraction in the displayed formula above is at most $1$ since to each multiple in the numerator we can correspond a bigger multiple in the denominator.

In what follows, $o(1)$ is with respect to $s\to \infty$. Keep in mind the fact that $(m/(m+2))^k=1-o(1)$ as $s\to \infty$ (independently of the behaviour of $k$). With \eqref{eqbincoeff} in hand, we conclude that the last expression in \eqref{eqest1} is at most
\begin{footnotesize}$${k\choose i+1} \frac{(i(s+1))^{i+1}(m+2-(i+1)(s+1))^{k-i-1}}{(m+1)^k}\le (1+o(1)) {k\choose i+1} \frac{(i(s+1))^{i+1}(m-(i+1)(s+1))^{k-i-1}}{m^k}.$$\end{footnotesize}
 Suppose that $m=c'(k-1)(s+1)$ with $c'>1$. Then \begin{equation}\label{eqbound}\frac{|\G_i|}{{m\choose k}}\le (1+o(1)){k\choose i+1}\frac{i^{i+1}(c'(k-1)-i-1)^{k-i-1}}{(c'(k-1))^k}=:\varphi(k,i,c').\end{equation}
  Taking the derivative in $c'$, it is easy to see that, as long as $c'>k/(k-1)$, the value $\varphi(k,i,c')$ decreases as $c'$ increases.
\vskip+0.2cm


\subsection{Proof of Lemma~\ref{lembeta}}\label{secbeta}
Recall that $c=1.666.$ As the statement of Lemma~\ref{lembeta} suggests, we have two possible ways to finish to the proof. First, by the induction hypothesis, the EMC holds for $m:=n-s-1=(c+\eps)s(k-1)-1$ and sets of size $k-1$. (Indeed, $n-s-1 > s+(c+\eps)s(k-2)$ and the induction hypothesis is applicable.) Then, using  \eqref{eqshadowmatch}, we infer $|\partial\ff(\emptyset)|\le {n-s-1\choose k-1}-{n-2s-1\choose k-1}=:\lambda{n-s-1\choose k-1},$ and the equation \eqref{cond1} holds if $\lambda\le \frac {s(c-1)}{q'c}$. Let us obtain a bound on $\lambda$. We have\footnote{Being more accurate in the first inequality below, the bound below can be improved to $\lambda\le 1- e^{-\frac {k-1}{c(k-1)-\tau}}$, where $\tau = \tau(k)<1$ is some fixed positive constant that tends to $\frac 12$ as $k\to\infty$. This improvement is helpful for getting better bounds, e.g., the ones given in Theorem~\ref{thm1v2}.}
\begin{align*}\lambda= &\frac{{m\choose k-1}-{m-s\choose k-1}}{{m\choose k-1}}= 1-\prod_{i=0}^{s-1}\frac{m-k-i+1}{m-i}\le 1-\Big(\frac{m-k+1-s}{m-s}\Big)^{s}\\ \le& 1-\Big(1-\frac{k-1}{(c+\eps/2)s(k-1)-s}\Big)^s \le 1- e^{-\frac {k-1}{c(k-1)-1}},\end{align*}
where in the last inequality we use that for any $a>b>0$ there exists $s_0$ such that $1-\frac b s\ge e^{-a/s}$ holds for any $s\ge s_0$. Thus, we are done if
\begin{equation}\label{eqcond1} 1- e^{-\frac {k-1}{c(k-1)-1   }}\le \frac {s(c-1)}{q'c},\end{equation}
provided that we know that the Erd\H os Matching Conjecture holds for this $c$ and $k-1$.  The displayed bound is effective if we get an upper bound on $q'$, which is significantly better than $s$.\vskip+0.1cm

If we cannot get a satisfactory bound on $q'$, then we argue that the family $\ff(\emptyset)$ should be somewhat small. To formalize this, we use the calculations from the previous subsection. Indeed, $\ff(\emptyset)$ satisfies the condition $\nu(\partial\ff(\emptyset))\le s$, and, putting $\G:=\ff(\emptyset)$, we can get a decomposition $\ff(\emptyset)=\bigsqcup_{i=1}^{k-1}\ff_i(\emptyset)$, in which $|\ff_i(\emptyset)|$ satisfies \eqref{eqbound} with $c':=c$ and $m$ as above. Indeed, we only need to note that $m=n-s-1=(c+\eps)(k-1)s-1\ge c(k-1)(s+1)$ for sufficiently large $s$.

 For each $i\in [k-1]$, define $\rho_i$,  $\rho_i\in[0,1+o(1))$ as follows: $$\rho_i:=\frac{|\ff_i(\emptyset)|}{{m\choose k} \varphi(k,i,c)}.$$
 If we have  $|\ff(\emptyset)|\ge q'|\partial\ff(\emptyset))|$ then, using \eqref{eqgoodshadow}, we get
\begin{equation}\label{eqcond2}\frac{\sum_{i=1}^{k-1}\rho_i\varphi(k,i,c)}{\sum_{i=1}^{k-1}\frac{i+1}{i} \rho_i\varphi(k,i,c)}\ge \frac{q'}s.\end{equation}
The ratio on the left hand side only increases if we replace the sets in $\ff_i(\emptyset)$ with sets in $\ff_j(\emptyset)$, where $j>i$. Thus, we may w.l.o.g. assume that there exists $i\in [k-1]$, such that $\rho_j=0$ for $j<i$ and $\rho_j=1+o(1)$ for $j>i$. Actually, in what follows we assume that $\rho_j=1$ for $j>i$ since it only alters the left hand side by a factor of $(1+o(1))$, and adjusting the value of $\delta$ (see below) compensates for it. Since \eqref{eqbound} is an inequality, such a choice of $\rho_i$ is probably not even possible, but it does not matter for our purposes. Keep in mind that we assume such a precise form of $\rho_1,\ldots, \rho_{k-1}$ for $\ff(\emptyset)$ when we  make statements involving $|\ff(\emptyset)|$. Note that the expression on the left hand side of \eqref{eqcond2} decreases as $|\ff(\emptyset)|$ increases.

We say that for a fixed $c, k$ and a sequence $\rho_1,\ldots, \rho_{k-1}$, the equations \eqref{eqcond1} and \eqref{eqcond2} are {\it $\sigma$-consistent} for some $\sigma\ge 0$, if the largest $q'$ satisfying \eqref{eqcond1} is bigger by $\sigma s$ than the largest $q'$ satisfying \eqref{eqcond2} with such $\rho_i$. Consistency implies that, for such $\rho_1,\ldots, \rho_{k-1}$, and thus for a fixed value of $|\ff(\emptyset)|$, as well as for larger $\ff(\emptyset)$, the corresponding value of $q'$  satisfies \eqref{eqcond1} with a certain margin.
In particular, Lemma~\ref{lembeta} is true for families of such size.

If the ratio on the left hand side of \eqref{eqcond2} is big, then, using \eqref{eqbound} and \eqref{eqcond2}, we can show that the family $\ff(\emptyset)$ is small, that is, it satisfies the inequality~\eqref{cond2}:
\begin{equation}\label{eqcond3} \sum_{i=1}^{k-1}\rho_i\varphi(k,i,c)\le \frac{(c-1)k}{c^2(k-1)}-\delta,
\end{equation}
where $\delta>0$ is a small constant, say, $10^{-6}$.
The intuition behind \eqref{eqcond2}, \eqref{eqcond3} is that, the bigger $q'$ in \eqref{eqcond2} is, the more members of the sequence $\rho_j$ are equal to $0$. But the upper bound \eqref{eqbound} on the size of $\ff_i(\emptyset)$ decreases exponentially in $i$, and thus eventually the equation \eqref{eqcond3} is satisfied.

We say that, for some $c>1$, $k\ge 4$ and $\sigma\ge 0$, the choice of $\rho_1,\ldots, \rho_{k-1}$ is {\it $\sigma$-robust} if the left hand side of the inequality \eqref{eqcond3} is smaller than the right hand side by at least $\sigma$.

To summarize, if we can find $\sigma>0$, such that, for a given $c$ and for all $k$, we can find a choice of $\rho_1,\ldots, \rho_{k-1}$ such that the equations \eqref{eqcond1} and \eqref{eqcond2} are $\sigma$-consistent and \eqref{eqcond3} is $\sigma$-robust, then we proved the lemma. Indeed, we have already mentioned that for larger $\ff(\emptyset)$ the lemma is valid since \eqref{eqcond1} is satisfied. Moreover, for smaller $\ff(\emptyset)$ the inequality \eqref{eqcond3} is satisfied.\\

Recall that $c=1.666.$ Using Wolfram Mathematica, it is easy to verify the following:
\begin{enumerate}
  \item[\textbf A] For $4\le k\le 10$, $\rho_1=\rho_2=0,\ \rho_3=0.6$ and $\rho_j=1$ for $4\le j\le k-1$, the equations \eqref{eqcond1} and \eqref{eqcond2} are $0.007$-consistent. Moreover,  \eqref{eqcond3} is $0.08$-robust.
  \item[\textbf B] For $10<k\le 2\cdot 10^4$, $\rho_1=\rho_2=0,\ \rho_3=\frac 18$ and $\rho_j=1$ for $4\le j\le k-1$, the equations \eqref{eqcond1} and \eqref{eqcond2} are $0.028$-consistent. Moreover,   \eqref{eqcond3} is $0.004$-robust, even with the right hand side replaced by $\frac{c-1}{c^2}-\delta$.
  \item[\textbf C] For $k=2\cdot 10^4$, $\rho_1=\rho_2=0,\ \rho_3=0.12$ and $\rho_j=1$ for $4\le j\le k-1$, we have $\sum_{i=1}^{k-1}\rho_i\varphi(k,i,c)\approx 0.235$ and $\sum_{i=1}^{k-1}\rho_i\varphi(k,i,c)/\sum_{i=1}^{k-1}\frac{i+1}i\rho_i\varphi(k,i,c)\approx 0.88$.
\end{enumerate}
Note that \textbf A and \textbf B alone verify the validity of Lemma~\ref{lembeta} for $k\le 2\cdot 10^4$. To verify the statement of the lemma for any $k>2\cdot 10^4$, we need to put some additional effort and to show that the situation is in some sense stable for larger $k$.

First of all, let us estimate $\varphi(k,i,c)$ for large $i$. For $i\ge 2$ we have
$$\varphi(k,i,c)= {k\choose i+1}\frac{i^{i+1}(c(k-1)-i-1)^{k-i-1}}{(c(k-1))^k}\le \frac{i^{i+1}e^{-\frac{(i+1)(k-i-1)}{c(k-1)}}}{(i+1)!c^{i+1}}\le \Big(\frac{e^{1-(k-i)/kc}}{c}\Big)^{i+1},$$
which, again via Mathematica calculations, implies that $\varphi(k,100,c)\le 3\cdot 10^{-5}$, and the right hand side of the expression above decreases at least as fast as a geometric progression with base $0.91$ until $i=k/40$. Thus, $\sum_{i=100}^{k/40}\varphi(k,100,c)\le 4\cdot 10^{-4}$. Moreover, for $i=k/40$ the value of the expression on the right hand side of the displayed equation is at most $e^{-k/500}$, and the function $k\cdot e^{-k/500}$ is decreasing for $k\ge 2\cdot 10^4$, with its maximum in $k=2\cdot 10^4$ being at most $10^{-13}$. Therefore, we obtain
\begin{equation}\label{stupidcalc1}\sum_{i=100}^{k-1}\varphi(k,i,c)\le \frac 1{2000},\end{equation}
provided that $\varphi(k,i,c)\le e^{-k/500}$ for any $i\ge k/40$, $k\ge 2\cdot 10^4$. The latter is not difficult to verify by comparing terms $\varphi(k,i,c)$ and $\varphi(k,i+1,c)$.

In what follows, we have to deal with the cases $3\le i\le 99$ (recall that $\rho_1=\rho_2=0$ for our choice of $\rho_i$ in \textbf B, \textbf C). Let us compare the terms $\varphi(k,i,c)$ and $\varphi(k+1,i,c)$.
We have
\begin{small}\begin{align*}\frac{\varphi(k+1,i,c)}{\varphi(k,i,c)}&=\frac{{k+1\choose i+1}\frac{i^{i+1}(ck-i-1)^{k-i}}{(ck)^{k+1}}}{{k\choose i+1}\frac{i^{i+1}(c(k-1)-i-1)^{k-i-1}}{(c(k-1))^k}}= \frac{k+1}{k-i}\Big(\frac{k-1}{k}\Big)^{i+1}\cdot \Big(\frac{(k-1)(ck-i-1)}{k(c(k-1)-i-1)}\Big)^{k-i-1}\frac{ck-i-1}{ck}\\
&=\Big(1+\frac{i+1}{k-i}\Big)\Big(1-\frac 1k\Big)^{i+1}\Big(1+\frac{i+1}{k(c(k-1)-i-1)} \Big)^{k-i-1}\Big(1-\frac{i+1}{ck}\Big)\\
&\le   e^{\frac{i+1}{k-i}}e^{-\frac{i+1}{k}}e^{\frac{(i+1)(k-i-1)} {ck(k-1-\frac{i+1}c)}}e^{-\frac{i+1}{ck}}\le e^{\frac{i(i+1)}{k(k-i)}}\le e^{\frac{(i+1)^2}{k^2}},\end{align*}
\end{small}
\noindent where in the second to last inequality we used that the product of the last two terms in the expression to the left is at most $1$ for $i\ge 3$, and in the last inequality we used the fact that $k>i(i+1)$ for $i\le 99,$ $k\ge 2\cdot 10^4$.

Thus, for any $k'>k\ge 2\cdot 10^4$, we have $$\frac{\varphi(k',i,c)}{\varphi(k,i,c)}\le e^{(i+1)^2\sum_{j=k}^{\infty}\frac{1}{j^2}}\le e^{1.01(i+1)^2/k}=:g(i,k',k).$$

Using this formula, we can estimate, how much do the sums of $\varphi$'s change. Let $k'>k:=2\cdot 10^4$. Then
\begin{equation}\label{stupidcalc2}\sum_{i=3}^{99}\varphi(k',i,c)-\varphi(k,i,c)\le \sum_{i=3}^{99}(g(i,k'k)-1)\frac{i^{i+1}e^{-\frac{(i+1)(k-i-1)}{c(k-1)}}}{(i+1)!c^{i+1}}\le 0.003,\end{equation}
where the calculation is again using Mathematica. Moreover,
 $\sum_{i=3}^{99}\frac {i+1}i(\varphi(k',i,c)-\varphi(k,i,c))\le 0.003$ as well.

Combining \eqref{stupidcalc1} and \eqref{stupidcalc2}, we can right away conclude, that from the $0.005$-robustness of \eqref{eqcond3} for $k=2\cdot 10^4$ with $\rho_i$ as in B we can infer the $0.0005$-robustness of \eqref{eqcond3} for the same $\rho_i$ and any $k'>2\cdot 10^4$.

Similarly, we claim that we can maintain the consistency of \eqref{eqcond1} and \eqref{eqcond2} for any $k'>2\cdot 10^4$. First, we note that the minimum $q'$ for which \eqref{eqcond1} holds increases as $k$ increases, and therefore it is sufficient to show that the maximal $q'$, for which \eqref{eqcond2} holds does not increase by more than, say, $0.002 s$ for any $k'>2\cdot 10^4$, as compared to its value for $k=2\cdot 10^4$. But, as we have seen before, the numerator on the left hand side of \eqref{eqcond2} increases by at most $\zeta := 0.0035$, and if the numerator increases by $a$, then the denominator increases by at least $a$. Note that $\frac{N+a}{D+a}-\frac ND = \frac{(D-N)a}{D(D+a)}$. Using \textbf C, the value of the expression on the left hand side of \eqref{eqcond2} for $k=2\cdot 10^4$ is $0.88,$ and the absolute value of the denominator is at least $0.23$. Substituting this into the last formula, we get that the fraction increases by at most $\zeta\cdot\frac{0.12}{0.23}<0.002$. The case when $\varphi(k+1,i,c)<\varphi(k,i,c)$ is treated similarly. Therefore, \eqref{eqcond1} and \eqref{eqcond2} are $0.0008$-consistent for any $k'>2\cdot 10^4$. The lemma is proven.

\end{document}